\documentclass[12pt,reqno]{amsart}
\usepackage{amsmath,amssymb,amsthm,mathtools,wasysym,calc,verbatim,enumitem,tikz,pgfplots,url,mathrsfs,fullpage}

\usepackage{comment}

\newtheorem{definition}{Definition}[section]
\newtheorem{remark}{Remark}[section]
\newtheorem{theorem}{Theorem}[section]
\newtheorem{conjecture}{Conjecture}[section]
\newtheorem{corollary}{Corollary}[section]
\newtheorem{lemma}{Lemma}[section]
\newtheorem{proposition}{Proposition}[section]

\theoremstyle{definition}

\theoremstyle{remark}

\usepackage[colorlinks = true,
            linkcolor  = blue,
            citecolor  = blue,
            urlcolor   = blue,
            bookmarks  = true]{hyperref}

\theoremstyle{remark}

\numberwithin{equation}{section}

\usepackage[backend=bibtex, style=numeric, citestyle=numeric-comp]{biblatex}
\DeclareFieldFormat{labelnumberwidth}{\mkbibbrackets{#1}}
\DeclareNameAlias{author}{family-given}

\DeclareFieldFormat*{title}{#1}

\AtEveryBibitem{%
  \clearfield{issn}
  \clearfield{url}
  \clearfield{eprint}
}

\renewbibmacro{in:}{}

\DeclareFieldFormat[article]{volume}{\textbf{#1}}
\DeclareFieldFormat[article]{number}{\textbf{#1}}
\renewbibmacro*{volume+number+eid}{%
  \printfield{volume}
  \iffieldundef{number}
    {}%
    {\printtext[parens]{\printfield{number}}}
  \setunit{\addcomma\space}%
  \printfield{eid}}

\DeclareFieldFormat[journal]{journaltitle}{\mkbibemph{#1}}

\DeclareNameAlias{author}{family-given}
\DeclareNameAlias{editor}{family-given}

\DeclareDelimFormat{namedelim}{\addcomma\space}
\DeclareFieldFormat[article,incollection,inproceedings,book]{giveninit}{\mkbibnamegiven{#1}}

\DeclareNameFormat{family-given}{%
  \usebibmacro{name:family-given}
    {\namepartfamily}
    {\namepartgiveni} 
    {\namepartprefix}
    {\namepartsuffix}%
  \usebibmacro{name:andothers}}

\DeclareFieldFormat{pages}{#1}

 \allowdisplaybreaks[4]

\begin{document}

\title{Chaining, tree and measure for some canonical processes}
\author{Xuanang Hu}
\address{Shandong University,  Jinan, China.}
\email{xuananghu7@gmail.com}

\author{Hanchao Wang}
\address{Shandong University, Jinan,  China.}
\email{wanghanchao@sdu.edu.cn}

\author{Xinglong Wu}
\address{Shandong University, Jinan,  China.}
\email{wuxinglong3@gmail.com}


\begin{abstract}
We prove two deterministic results for families of distances arising in the
study of canonical processes. The first derives an admissible partition
scheme from a growth condition. The second gives a representation in terms
of parameterized separation trees and compares it with the corresponding
majorizing-measure quantities. The main point is that the proofs do not
depend on the distribution of the underlying process: once the initial
distance and the family of distances are given, no random variables,
independence, tail functions, or moment estimates are used.

For canonical processes with regular log-concave tails, the assumptions of
the abstract results follow from the usual regularity conditions. One
direction of the separation-tree estimate also applies to Bernoulli
processes without these additional assumptions, and we prove the reverse
estimate for bounded convex unconditional index sets. We also give a version of the growth argument for points which, for finite
index sets, leads to a recursive construction of admissible partitions and
parameterized separation trees.
\end{abstract}

\maketitle

\section{Introduction and main results}
\label{Introduction}

Let $(X_i)_{i\ge1}$ be independent centered random variables. For
$t=(t_i)_{i\ge1}$ such that the series converges in $L^1$, set
\[
X_t:=\sum_{i\ge1}t_iX_i.
\]
For an index set $T$ consisting of such points, we consider the canonical
process $(X_t)_{t\in T}$. To avoid measurability questions, set
\[
\textsf{E}\sup_{t\in T}X_t
:=
\sup\left\{
\textsf{E}\sup_{t\in F}X_t:
F\subset T,\ F\text{ finite}
\right\}.
\]

We are interested in estimates of this quantity in terms of the geometry of
the index set. For Gaussian processes this is the content of the
majorizing-measure theorem. Talagrand proved the converse majorizing-measure
estimate and obtained a characterization of the expected supremum in terms
of the canonical metric; see \cite{tala1,talab}. He later considered
canonical processes generated by independent variables with densities
proportional to
\[
\exp(-|x|^\alpha),
\qquad 1\le\alpha<\infty,
\]
and introduced families of distances adapted to the different scales of the
process; see \cite{tala2}.

Lata{\l}a obtained a Sudakov minoration and a majorizing-measure
characterization for canonical processes with independent log-concave-tailed
coordinates under an upper regularity assumption on the tails
\cite{lata1}. Lata{\l}a and Tkocz later gave a characterization by chains of
increment moments for variables with regular moments \cite{lata2}. We refer
to \cite{talab} for the corresponding partition and growth constructions.

For Bernoulli processes, Talagrand conjectured that the expected supremum is
described by a decomposition into an $\ell_1$ part and a Gaussian chaining
part. This was proved by Bednorz and Lata{\l}a \cite{bl}. Related partition
and decomposition arguments also occur for infinitely divisible processes
\cite{b1}, while tree constructions for stochastic suprema appear, for
example, in \cite{gztree}.

A representation in terms of parameterized separation trees for canonical
processes with regular log-concave tails was obtained in
\cite{hu2025dual}. In that setting, the family of distances is constructed
from the tails of the underlying random variables. Here we separate the
partition and separation-tree arguments from this probabilistic origin. Our
results are formulated directly for an initial distance and a family
$(\varphi_j)_{j\in\mathbb Z}$ that satisfy the conditions below. Once these
objects are fixed, no random variables occur in the proofs of
Theorems~\ref{m-1} and~\ref{m-2}. In particular, independence, tail
functions, and moment estimates are used only when the abstract assumptions
are verified in a particular application.

The first assertion of Theorem~\ref{m-2} requires only the basic assumptions
on the family of distances and therefore also applies to the Bernoulli
family. We ask whether the reverse inequality holds for arbitrary Bernoulli
index sets and prove it for bounded convex unconditional subsets of
$\mathbb R^d$.

We also prove a pointwise form of the partition theorem. The set functionals
in the growth condition are replaced by functions indexed by a point and the
two scale parameters. For finite $T$ this gives a finite recursion based on
greedy separated families; the points selected by the recursion also define
a parameterized separation tree.

We now introduce the quantities used in the two main theorems.

We begin with admissible partitions. An admissible sequence of $T$ is an
increasing sequence $(\mathcal A_n)_{n\ge0}$ of partitions of $T$ such that
\[
|\mathcal A_0|=1=:N_0,
\qquad
|\mathcal A_n|\le N_n:=2^{2^n},
\qquad n\ge1.
\]
We shall also call an admissible sequence a \emph{partition scheme}.

We use the following notation for comparisons. For non-negative quantities $a$ and $b$, the relation
\[
a\lesssim b
\]
means that $a\le Cb$ for a numerical constant $C$. We write
$a\asymp b$ if both $a\lesssim b$ and $b\lesssim a$ hold. If the constant is
allowed to depend on a parameter, this parameter is indicated by a subscript,
thus $ a\lesssim_r b$
means $a\le C(r)b$. The notation $\gtrsim$ and $\asymp_r$ is used in the same
way. 

Let $E$ be a real vector space and let $T\subset E$ be non-empty. Consider a
family
\[
\varphi=(\varphi_j)_{j\in\mathbb Z},
\qquad
\varphi_j:E\times E\longrightarrow[0,\infty],
\]
such that
\begin{equation}\label{eq:dj(s,t)=dj(t,s)}
\begin{gathered}
\varphi_j(t,t)=0,\qquad
\varphi_j(s,t)=\varphi_j(t,s),\qquad
\varphi_{j+1}(s,t)\ge\varphi_j(s,t),\\
\varphi_j(s+a,t+a)=\varphi_j(s,t),
\qquad s,t,a\in E,\ j\in\mathbb Z.
\end{gathered}
\end{equation}
We shall also use
\begin{equation}\label{eq:dj<dj+dj}
\varphi_j(s,t)
\le
2\bigl(\varphi_j(s,u)+\varphi_j(u,t)\bigr),
\qquad s,t,u\in E.
\end{equation}
We refer to $(\varphi_j)_{j\in\mathbb Z}$ as a family of distances.

Fix an integer $\kappa\ge5$ and put
\[
r=2^{\kappa-3}.
\]

We use an additional distance $d_0$ at level $0$ and assume that
\begin{equation}\label{eq:initial-distance-compatibility}
d_0(s,t)\le r^{-j}
\quad\Longrightarrow\quad
\varphi_j(s,t)\le1,
\qquad s,t\in T,\ j\in\mathbb Z.
\end{equation}
For every non-empty $A\subset T$, put
\[
\operatorname{diam}_{d_0}(A)
:=
\sup_{s,t\in A}d_0(s,t)
\]
and define
\[
j_0(A)
:=
\sup\left\{
j\in\mathbb Z:
\operatorname{diam}_{d_0}(A)\le r^{-j}
\right\}
\in\mathbb Z\cup\{-\infty,\infty\}.
\]
Thus
\[
j_0(A)=\infty
\quad\text{if }\operatorname{diam}_{d_0}(A)=0,
\]
whereas
\[
j_0(A)=-\infty
\quad\text{if }\operatorname{diam}_{d_0}(A)=\infty.
\]
We put
\[
r^{-\infty}=0,
\qquad
r^{-(-\infty)}=\infty.
\]
For $A=T$, write $j_0=j_0(T)$.

Let $\mathcal A=(\mathcal A_n)_{n\ge0}$ be a partition scheme. For
$n\ge1$ and $A\in\mathcal A_n$, define
\[
j_n^\varphi(A)
:=
\sup\left\{
j\in\mathbb Z:
\varphi_j(s,t)\le2^n
\text{ for all }s,t\in A
\right\}.
\]
We set
\[
\sup\varnothing=-\infty,
\qquad
r^{-(-\infty)}=\infty.
\]
If $A_n(t)$ denotes the member of $\mathcal A_n$ containing $t$, put
\begin{equation}\label{eq:gamma-functional}
\operatorname{Val}_{\varphi}(\mathcal A)
:=
r^{-j_0}
+
\sup_{t\in T}
\sum_{n\ge1}
2^n r^{-j_n^\varphi(A_n(t))}.
\end{equation}
For every non-empty $A\subset T$, define
\[
\gamma_\varphi(A)
:=
\inf_{\mathcal A}
\left[
r^{-j_0(A)}
+
\sup_{t\in A}
\sum_{n\ge1}
2^nr^{-j_n^\varphi(A_n(t))}
\right],
\]
where the infimum is taken over all partition schemes of $A$.

For canonical processes, the family $(\varphi_j)$ is chosen from the
increment geometry of the process. The partition functionals arising in this
way occur in the log-concave, Bernoulli and infinitely divisible settings
discussed above. In what follows we use only the deterministic properties of
$(\varphi_j)$.

For $t\in E$, $c\ge0$, and $j\in\mathbb Z$, set
\[
B_j(t,c)
:=
\{s\in E:\varphi_j(t,s)\le c\},
\qquad
B_j(c):=B_j(0,c).
\]
By translation invariance,
\[
B_j(t,c)=t+B_j(c).
\]
For $a>0$, write
\[
aB_j(t,c)
:=
t+a\bigl(B_j(t,c)-t\bigr).
\]

In the partition construction we shall use the following two relations:
\begin{equation}\label{eq:ball-arithmetic}
aB_j(c)+bB_j(c)
\subset
(a+b)B_j(c),
\qquad a,b\ge0,
\end{equation}
and
\begin{equation}\label{mild condition}
B_{j+1}(2^{n+\kappa})
\subset
\frac1{20}B_j(2^n),
\qquad
j\in\mathbb Z,\ n\ge0.
\end{equation}
The second inclusion is used when the recursive construction changes both
$n$ and $j$. In Section~\ref{5} it follows from the regularity assumptions
on the coordinate tails.

We now state the growth condition. It is the local estimate used to construct
the partition scheme.

\begin{definition}\label{groth con}
Let $(F_n)_{n\ge0}$ be functionals on the subsets of $T$ such that
\[
A\subset C
\quad\Longrightarrow\quad
0\le F_n(A)\le F_n(C)
\]
and
\[
F_{n+1}(A)\le F_n(A).
\]
Fix $j\in\mathbb Z$ and $n\ge1$, and put
\[
m=N_n/2.
\]
Suppose that $H_1,\ldots,H_m$ are non-empty subsets of $T$ and that there
exist $u,t_1,\ldots,t_m\in T$ satisfying
\[
H_\ell\subset4B_{j+2}(t_\ell,2^{n+\kappa}),
\qquad
t_\ell\in B_j(u,2^n),
\]
and
\[
\varphi_{j+1}(t_\ell,t_{\ell'})
>
2^{n+1},
\qquad
\ell\ne\ell'.
\]
We say that $(F_n)$ satisfies the \emph{growth condition} if
\[
F_n\left(\bigcup_{\ell=1}^mH_\ell\right)
\ge
2^nr^{-j-1}
+
\min_{1\le\ell\le m}F_{n+1}(H_\ell).
\]
\end{definition}

Our first theorem is the partition construction.

\begin{theorem}\label{m-1}
Assume \eqref{eq:initial-distance-compatibility},
\eqref{eq:dj(s,t)=dj(t,s)}, \eqref{eq:dj<dj+dj},
\eqref{eq:ball-arithmetic}, and \eqref{mild condition}, and suppose that
$j_0\in\mathbb Z$. Let $(F_n)$ satisfy the growth condition. Then there exist a
partition scheme $(\mathcal A_n)_{n\ge0}$, integers $q_n(A)$, and points
$t_{n,A}\in T$ such that, whenever
\[
A\in\mathcal A_n,\qquad
C\in\mathcal A_{n-1},\qquad
A\subset C,
\]
we have
\[
q_{n-1}(C)
\le
q_n(A)
\le
q_{n-1}(C)+1,
\]
\[
A
\subset
B_{q_n(A)}(t_{n,A},2^n),
\]
and
\[
\sup_{t\in T}
\sum_{n\ge0}
2^nr^{-q_n(A_n(t))}
\lesssim_r
F_0(T)+r^{-j_0}.
\]
Consequently,
\[
\gamma_\varphi(T)
\lesssim_r
F_0(T)+r^{-j_0}.
\]
\end{theorem}

Theorem~\ref{m-1} will also be applied to functionals defined from
parameterized separation trees. This application is used in the proof of the
second assertion of Theorem~\ref{m-2}.

We next introduce the quantity used in the second theorem. The full definition
of a parameterized separation tree is given in Section~\ref{4}. We record
here the notation needed for its value.

Let $\mathcal T$ be a finite parameterized separation tree with non-empty
root $H$. For a node $A$, write $c(A)$ for its set of children, and put
\[
S_{\mathcal T}
:=
\bigcup\{A\in\mathcal T:c(A)=\emptyset\}.
\]
Every non-root node $A$ carries two parameters
\[
\mathbf n(A)\in\{1,2,\ldots\},
\qquad
\mathbf j(A)\in\mathbb Z\cup\{\infty\}.
\]
Define
\[
\operatorname{Size}(\mathcal T)
:=
r^{-j_0(H)}
+
\inf_{t\in S_{\mathcal T}}
\sum_{\substack{
A\in\mathcal T\setminus\{H\}\\
t\in A}}
2^{\mathbf n(A)}r^{-\mathbf j(A)}.
\]
Finally, put
\begin{equation}\label{eq:rho-T}
\rho_\varphi(T)
:=
\sup\left\{
\operatorname{Size}(\mathcal T):
\mathcal T\text{ is a parameterized separation tree on a non-empty }
H\subset T
\right\}.
\end{equation}

We shall compare $\rho_\varphi$ with the following measure quantity. Let
$\mu$ be a probability measure on $T$ with countable support. For $n\ge1$,
set
\[
k_n^\mu(t)
:=
\sup\left\{
j\in\mathbb Z:
\mu\bigl(B_j(t,2^n)\bigr)\ge N_n^{-1}
\right\},
\]
where
\[
\sup\varnothing=-\infty.
\]
Put
\[
I_{\mu,\varphi}(t)
:=
r^{-j_0}
+
\sum_{n\ge1}
2^nr^{-k_n^\mu(t)}
\]
and define
\begin{equation}\label{eq:M-intro}
\mathcal M_\varphi(T)
:=
\inf_\mu
\sup_{t\in T}
I_{\mu,\varphi}(t),
\end{equation}
where the infimum is taken over all probability measures on $T$ with
countable support.

\begin{theorem}\label{m-2}
Assume \eqref{eq:initial-distance-compatibility},
\eqref{eq:dj(s,t)=dj(t,s)}, and \eqref{eq:dj<dj+dj}. Then
\[
\rho_\varphi(T)
\lesssim_r
\mathcal M_\varphi(T).
\]
If, in addition, \eqref{eq:ball-arithmetic} and
\eqref{mild condition} hold, then
\[
\rho_\varphi(T)
\asymp_r
\mathcal M_\varphi(T)
\asymp_r
\gamma_\varphi(T).
\]
\end{theorem}

The two inequalities relating $\rho_\varphi$ and
$\mathcal M_\varphi$ are proved separately. The inequality
\[
\rho_\varphi(T)\lesssim_r\mathcal M_\varphi(T)
\]
follows directly from the separation of distinct children and a counting
argument with an arbitrary probability measure. For the reverse inequality,
Section~\ref{4} associates a sequence of set functionals $(F_n)_{n\ge0}$
with separation trees. Joining trees below a separated family gives the
growth condition of Definition~\ref{groth con}. Theorem~\ref{m-1} then gives
\[
\gamma_\varphi(T)\lesssim_r\rho_\varphi(T).
\]
Proposition~\ref{equivalence} then gives
\[
\mathcal M_\varphi(T)\lesssim_r\rho_\varphi(T).
\]

There is a further consequence for Bernoulli processes. Let
$(\varepsilon_i)_{i\ge1}$ be independent signs and set
\[
\mathcal E_t=\sum_{i\ge1}t_i\varepsilon_i,
\qquad
b(T)=\textsf E\sup_{t\in T}\mathcal E_t.
\]
For the Bernoulli family
\[
\varphi_j^{\mathrm B}(s,t)
=
\sum_{i\ge1}
\bigl((r^j(s_i-t_i))^2\wedge1\bigr),
\]
Combining Theorem~\ref{m-2} with
Theorem~\ref{bernoulli partition}, we obtain
\[
\rho_{\varphi^{\mathrm B}}(T)
\lesssim_r
b(T)
\]
for arbitrary $T\subset\ell^2$. We ask whether the reverse estimate
\[
b(T)
\lesssim_r
\rho_{\varphi^{\mathrm B}}(T)
\]
holds in general. In Section~\ref{6} we prove it when
$T\subset\mathbb R^d$ is bounded, convex and unconditional. Hence
\[
b(T)\asymp_r\rho_{\varphi^{\mathrm B}}(T)
\]
for this class.

For finite $T$, The growth condition for points gives a direct recursion. The
set variable $A$ is replaced by a point $t$, while the two indices $n$ and
$j$ are retained. At each step a maximal separated family is chosen
greedily. The points selected in this construction also define a
parameterized separation tree. Section~\ref{7} shows that, for finite $T$, only finitely many values of the
index $j$ need to be considered. Under the assumptions stated there, the
recursion can be evaluated in polynomial time.

The rest of the paper is organized as follows.
Section~\ref{preliminaries} recalls the results used later and proves the
equivalence between the partition and measure formulations.
Section~\ref{3} proves Theorem~\ref{m-1} and its pointwise version.
Section~\ref{4} introduces parameterized separation trees and proves
Theorem~\ref{m-2}. Sections~\ref{5} and~\ref{6} give the applications
described above. Section~\ref{7} contains the finite-set construction.

\section{Classical results and basic equivalences}
\label{preliminaries}

We recall the Gaussian and Bernoulli examples and then compare the partition and measure functionals.

By \textit{universal} or \textit{absolute} constants we mean numbers independent of all parameters. Throughout the paper, $L$ denotes a universal constant whose value may change from one occurrence to the next.

We say that $F$ is a functional on $T$ if $F$ maps the subsets of $T$ to $[0,\infty]$ and is non-decreasing under inclusion: for every $A\subset B\subset T$, we have $F(A)\le F(B)$.

We first consider the Gaussian case.

Let $G_t=\sum\limits_{i\ge1}t_i g_i$ for $t\in\ell^2$, where $(g_i)_{i\ge1}$ is a sequence of independent and identically distributed (i.i.d.) standard Gaussian random variables. Then $(G_t)_{t\in T}$ is the canonical Gaussian process over $T\subset\ell^2$. Its natural metric is
$$d(s,t)=(\textsf{E}(G_s-G_t)^2)^{1/2}=\|t-s\|_2.$$
The corresponding family of distances is
\[
\varphi_j(s,t)=r^{2j}\|s-t\|_2^2, \qquad s,t\in T.
\]
We take $d_0(s,t)=\|s-t\|_2$. Since
\[
\varphi_j(s,t)=r^{2j}d_0(s,t)^2,
\]
condition~\eqref{eq:initial-distance-compatibility} holds.
We write $\gamma_G$ for this functional.

\begin{theorem}[Talagrand \cite{talab}]
    $$\gamma_G(T)\asymp_r\textsf{E}[\sup\limits_{t\in T} G_t].$$
\end{theorem}
Here the family $(\varphi_j)$ is determined by the single metric $d$.

Talagrand's measure formulation is the following; see \cite{tala1}.

For a probability measure $\mu$ on $T$, define
\[
I_\mu(x)=\int_0^\infty \sqrt{\log \frac{1}{\mu(B(x,\varepsilon))}}\, d\varepsilon,
\qquad x\in T,
\]
where
\[
B(x,\varepsilon)=\{t\in T:\ d(x,t)\le \varepsilon\}.
\]

One has

\begin{theorem}[Talagrand \cite{tala1}]
\[
\textsf{E}\sup_{t\in T} G_t \asymp \inf_{\mu}\sup_{t\in T} I_\mu(t),
\]
where the infimum is taken over all probability measures on $T$. Moreover, the infimum may be restricted to probability measures with countable support.
\end{theorem}

We next consider Bernoulli processes.

Let
\[
\mathcal{E}_t=\sum_{i\ge1} t_i \varepsilon_i, \qquad t\in \ell^2,
\]
where $(\varepsilon_i)_{i\ge1}$ are independent Rademacher variables, that is,
\[
\textsf{P}(\varepsilon_i=1)=\textsf{P}(\varepsilon_i=-1)=\frac12.
\]
Then the process \((\mathcal{E}_t)_{t\in T}\), with \(T\subset \ell^2\), is called the Bernoulli process over \(T\).

The Bednorz--Lata{\l}a theorem gives the following decomposition formula. For $T^1,T^2\subset\ell^2$, write

\[
T^1+T^2:=\{t^1+t^2:\ t^1\in T^1,\ t^2\in T^2\}.
\]

\begin{theorem}[Bednorz--Latała \cite{bl}]
\label{latala-bednortz}
   $$\textsf{E}[\sup\limits_{t\in T} \mathcal{E}_t ]\asymp \inf\{\gamma_G(T^2)+\sup\limits_{t^1\in T^1}||t^1||_1:T\subset T^1+T^2\}.
   $$ 
\end{theorem}
For Bernoulli processes, we can define the corresponding family of distances as follows:
\begin{equation}\label{bernoulli: family of distances}
    \varphi_j(s,t)=\sum\limits_{i\ge1}\bigl((r^j(s_i-t_i))^2\wedge1\bigr),\qquad s,t\in\ell^2.
\end{equation}
\begingroup
For this family, we take
\[
 d_0(s,t)=\|s-t\|_2.
\]
If $d_0(s,t)\le r^{-j}$, then
\[
 \varphi_j(s,t)\le r^{2j}\|s-t\|_2^2\le1.
\]
Thus $j_0$ is determined by the Euclidean diameter of $T$. For $n\ge1$, the numbers $j_n^\varphi(A)$ are determined by \eqref{bernoulli: family of distances}.
\endgroup
To emphasize the dependence of $\varphi$ on $\mathcal{E},$ we write 
$\gamma_\mathcal{E}$ instead of $\gamma_\varphi$.
The Bernoulli conjecture was proved by Bednorz and Lata{\l}a
\cite{bl}. We shall use the following partition formulation; see also
Talagrand \cite{talab}.

\begin{theorem}[Bernoulli theorem, partition form]
\label{bernoulli partition}
\[
\textsf{E}\sup_{t\in T}\mathcal E_t
\asymp_r
\gamma_{\mathcal E}(T).
\]
\end{theorem}

Theorem~\ref{bernoulli partition} will be used in Section~\ref{5}; Section~\ref{6} proves the reverse inequality for bounded convex
unconditional Bernoulli index sets.

\begingroup
We now define the measure functional associated with $(d_0,\varphi)$.

Fix a non-empty set $T\subset E$, and let $d_0$ and $(\varphi_j)_{j\in\mathbb Z}$ satisfy \eqref{eq:initial-distance-compatibility}. Let $\mu$ be a probability measure on $T$. For $n\ge1$, define
\begin{equation}\label{definition of jn}
k_n^\mu(t)
:=
\sup\{j\in\mathbb Z:\mu(B_j(t,2^n))\ge N_n^{-1}\},
\end{equation}
where
\[
\sup\varnothing=-\infty.
\] Since $\mu$ is supported on $T$,
we use $B_j(t,2^n)\cap T$ implicitly in this definition.
We set
\[
 I_{\mu,\varphi}(t):=r^{-j_0}+\sum_{n\ge1}2^nr^{-k_n^\mu(t)}
\]
and
\[
 \mathcal M_\varphi(T):=\inf_\mu\sup_{t\in T}I_{\mu,\varphi}(t),
\]
where the infimum is taken over all probability measures on $T$ with countable support.

The first term is determined by $d_0$ and the remaining terms by $(\varphi_j)$.
\endgroup 

We shall use the following comparison between the partition and measure
functionals.

\begin{proposition}
\label{equivalence}
Let $T\subset E$ be non-empty. Let $d_0$ and $(\varphi_j)_{j\in\mathbb Z}$ satisfy \eqref{eq:initial-distance-compatibility}. Assume that $(\varphi_j)$ satisfies \eqref{eq:dj(s,t)=dj(t,s)} and \eqref{eq:dj<dj+dj}. Then
\[
\gamma_\varphi(T)\asymp_r\mathcal M_\varphi(T).
\]
\end{proposition}

\begin{proof}
If $j_0=-\infty$, then
\[
r^{-j_0}=\infty,
\]
and hence
\[
\gamma_\varphi(T)=\mathcal M_\varphi(T)=\infty.
\]
If $j_0=\infty$, then
\eqref{eq:initial-distance-compatibility} gives
$\varphi_j(s,t)\le1$ for every finite $j$ and every $s,t\in T$, so both
functionals vanish. Thus we may assume that $j_0\in\mathbb Z$.

We first prove
$\mathcal M_\varphi(T)\lesssim_r\gamma_\varphi(T)$. Fix $\varepsilon>0$ and
choose a partition scheme $(\mathcal A_n)$ such that
\[
 r^{-j_0}
 +\sup_{t\in T}\sum_{n\ge1}2^nr^{-j_n^\varphi(A_n(t))}
 \le\gamma_\varphi(T)+\varepsilon.
\]
If the right-hand side is infinite, there is nothing to prove. For every
$n\ge1$ and $A\in\mathcal A_n$, choose $t_{n,A}\in A$ and assign mass
$1/(2^nN_n)$ to this point. The total assigned mass is at most
\[
 \sum_{n\ge1}\frac{|\mathcal A_n|}{2^nN_n}
 \le\sum_{n\ge1}2^{-n}=1.
\]
Put any remaining mass at one fixed point of $T$ and call the resulting
probability measure $\mu$.

Fix $t\in T$ and $n\ge1$. Since the chosen partition has finite value whenever
the right-hand side above is finite, $j_n^\varphi(A_n(t))>-\infty$. If this
index is finite, then
\[
 \mu\bigl(B_{j_n^\varphi(A_n(t))}(t,2^n)\bigr)
 \ge\mu(A_n(t))
 \ge\frac1{2^nN_n}
 \ge\frac1{N_{n+1}},
\]
and hence
\[
 k_{n+1}^\mu(t)\ge j_n^\varphi(A_n(t)).
\]
If $j_n^\varphi(A_n(t))=\infty$, the same conclusion holds with every finite
$j$ in place of $j_n^\varphi(A_n(t))$, and therefore
$k_{n+1}^\mu(t)=\infty$.

By \eqref{eq:initial-distance-compatibility},
$B_{j_0}(t,2)\supset T$, so $k_1^\mu(t)\ge j_0$. Consequently,
\begin{align*}
 I_{\mu,\varphi}(t)
 &\le3r^{-j_0}
 +\sum_{n\ge2}2^nr^{-j_{n-1}^\varphi(A_{n-1}(t))}\\
 &\le3r^{-j_0}
 +2\sum_{n\ge1}2^nr^{-j_n^\varphi(A_n(t))}.
\end{align*}
Taking the supremum over $t$ and then letting $\varepsilon\downarrow0$ proves
one inequality.

The reverse inequality follows from Theorem~11.6.3 of \cite{talab},
with the initial index $j_0$ fixed by $d_0$. The term corresponding to
$n=0$ is $r^{-j_0}$, while the remaining terms are those in
\eqref{eq:gamma-functional}.
\end{proof}

\section{Partition schemes and the proof of Theorem~\ref{m-1}}\label{3}

We begin with the ball relations used in this section and in the finite-set
construction.

\begin{lemma}\label{lem:ball-calculus}
Assume \eqref{eq:ball-arithmetic} and \eqref{mild condition}. Then the
following assertions hold.
\begin{enumerate}[label=(\roman*)]
\item If $s,t\in B_j(u,2^n)$, then
\[
4B_{j+2}(s,2^{n+\kappa})\subset 4B_j(t,2^n).
\]
\item If
\[
\varphi_{j+1}(s,t)>2^{n+1},
\]
then
\[
\varphi_{j+2}\bigl(4B_{j+2}(s,2^{n+\kappa}),
4B_{j+2}(t,2^{n+\kappa})\bigr)>2^n.
\]
\item If $d\ge0$ and $0\le m-n\le\kappa d$, then
\[
4B_{j+d}(t,2^m)\subset 4B_j(t,2^n).
\]
\item If
\[
x\notin 9B_{j+2}(y,2^{n+\kappa}),
\]
then
\[
\varphi_{j+2}\bigl(4B_{j+2}(x,2^{n+\kappa}),
4B_{j+2}(y,2^{n+\kappa})\bigr)>2^n.
\]
\item If
\[
\varphi_{j+1}(t,t')>2^{n+1},
\]
then every set of the form $9B_{j+2}(x,2^{n+\kappa})$ contains at most one
of the two points $t,t'$.
\item If $x\in 2B_j(u,2^n)$, then
\[
4B_{j+2}(x,2^{n+\kappa})\subset 4B_j(u,2^n).
\]
\end{enumerate}
\end{lemma}

\begin{proof}
The balls $B_j(c)$ are symmetric. Since $0\in B_j(c)$,
\eqref{eq:ball-arithmetic} also gives
\begin{equation}\label{eq:radial-stability-derived}
0\le\lambda\le1
\quad\Longrightarrow\quad
\lambda B_j(c)\subset B_j(c).
\end{equation}
Applying \eqref{mild condition} with $j+1$ in place of $j$, we obtain
\begin{equation}\label{eq:fine-ball-contraction}
4B_{j+2}(2^{n+\kappa})
\subset \frac15 B_{j+1}(2^n)
\subset \frac15 B_j(2^n).
\end{equation}
If $s,t\in B_j(u,2^n)$, then
$s-t\in2B_j(2^n)$. Hence
\[
4B_{j+2}(s,2^{n+\kappa})-t
\subset 2B_j(2^n)+\frac15B_j(2^n)
\subset4B_j(2^n),
\]
which proves (i).

For (ii), suppose that
$x\in4B_{j+2}(s,2^{n+\kappa})$ and
$y\in4B_{j+2}(t,2^{n+\kappa})$ satisfy
$\varphi_{j+2}(x,y)\le2^n$. Then
$x-y\in B_{j+2}(2^n)\subset B_{j+2}(2^{n+\kappa})$, and therefore
\[
s-t=(s-x)+(x-y)+(y-t)
\in9B_{j+2}(2^{n+\kappa}).
\]
By \eqref{mild condition},
\[
9B_{j+2}(2^{n+\kappa})
\subset\frac9{20}B_{j+1}(2^n)
\subset B_{j+1}(2^n),
\]
contrary to the assumption on $s,t$.

For (iii), monotonicity in the second argument and $m-n\le\kappa d$ give
\[
B_{j+d}(2^m)
\subset B_{j+d}(2^{n+\kappa d}).
\]
Iterating \eqref{mild condition} $d$ times gives
\[
B_{j+d}(2^{n+\kappa d})
\subset20^{-d}B_j(2^n)
\subset B_j(2^n).
\]
Multiplying the inclusion by $4$ proves (iii).

For (iv), suppose, on the contrary, that there exist
$z,z'$ in the two indicated sets such that
\[
z-z'\in B_{j+2}(2^n).
\] Hence
\[
x-y=(x-z)+(z-z')+(z'-y)
\in9B_{j+2}(2^{n+\kappa}),
\]
a contradiction.

If $t,t'$ belonged to the same set
$9B_{j+2}(x,2^{n+\kappa})$, then
\[
t-t'\in18B_{j+2}(2^{n+\kappa})
\subset\frac{18}{20}B_{j+1}(2^n)
\subset B_{j+1}(2^n),
\]
which proves (v). Finally, (vi) follows from
\eqref{eq:fine-ball-contraction} in the same way as (i):
\[
4B_{j+2}(x,2^{n+\kappa})-u
\subset2B_j(2^n)+\frac15B_j(2^n)
\subset4B_j(2^n).
\]
\end{proof}

\subsection{The growth condition}

Assume the hypotheses of Theorem~\ref{m-1}. Set
\[
j_0(t)=j_1(t)=j_0,
\qquad t\in T.
\]
Suppose that $j_k(t)$ has been defined for $k\le n$, where $n\ge1$. Put
\begin{align*}
M_n(t):=\Bigl\{m\ge n:\;&
F_n\bigl(4B_{j_n(t)}(t,2^n)\cap T\bigr)\\
&\ge 2^{m-1}r^{-j_n(t)-1}
+F_{n+1}\bigl(4B_{j_n(t)+2}(t,2^{m+\kappa})\cap T\bigr)
\Bigr\}.
\end{align*}
Define
\[
j_{n+1}(t)=
\begin{cases}
j_n(t)+1,&n\notin M_n(t),\\
j_n(t),&n\in M_n(t).
\end{cases}
\]
Then $0\le j_{n+1}(t)-j_n(t)\le1$. If $j_{n+1}(t)=j_n(t)$, then
\begin{equation}\label{eq:recursive-payment}
\begin{aligned}
&2^{n-1}r^{-j_n(t)-1}
+F_{n+1}\bigl(4B_{j_n(t)+2}(t,2^{n+\kappa})\cap T\bigr)\\
&\hspace{35mm}\le
F_n\bigl(4B_{j_n(t)}(t,2^n)\cap T\bigr).
\end{aligned}
\end{equation}

We shall use Talagrand's  lemma.

\begin{lemma}[Talagrand {\cite[Lemma 2.9.5]{talab}}]\label{tala lem}
Let $(a_k)$ be a finite or infinite bounded sequence of positive numbers and
let $\alpha>1$. Put
\[
I:=\{k:a_\ell<a_k\alpha^{|k-\ell|}\text{ for every }\ell\ne k\}.
\]
Then $I\ne\emptyset$ and
\[
\sum_k a_k\le\frac{2\alpha}{\alpha-1}\sum_{k\in I}a_k.
\]
\end{lemma}

\begin{proposition}\label{inequality-global}
We have
\[
\sup_{t\in T}\sum_{n\ge0}2^nr^{-j_n(t)}
\lesssim_r F_0(T)+r^{-j_0}.
\]
\end{proposition}

\begin{proof}
Fix $t\in T$ and write
\[
j(n):=j_n(t),\qquad a(n):=2^nr^{-j(n)}.
\]
Since $j(0)=j(1)=j_0$,
\[
a(0)+a(1)=3r^{-j_0}.
\]
The sequence $(a(n))_{n\ge1}$ is bounded. If
$j(n)=j(n-1)+1$, then
\[
a(n)=\frac2r a(n-1)\le\frac12a(n-1).
\]
If $j(n)=j(n-1)$ and $n\ge2$, then $n-1\in M_{n-1}(t)$ and
\eqref{eq:recursive-payment} gives
\[
a(n)\le4rF_0(T).
\]

Apply Lemma~\ref{tala lem}, with $\alpha=2$, to $(a(n))_{n\ge1}$.
Every selected index $n>1$ satisfies
\[
j(n-1)=j(n),\qquad j(n+1)=j(n)+1.
\]
Hence, with
\[
I_1:=\{1\}\cup\{n>1:j(n-1)=j(n),\ j(n+1)=j(n)+1\},
\]
we have
\begin{equation}\label{eq:first-selection}
\sum_{n\ge1}a(n)\le4\sum_{n\in I_1}a(n).
\end{equation}
Enumerate
\[
I_1=\{n_0,n_1,\ldots\},\qquad 1=n_0<n_1<\cdots,
\]
where the sequence is finite if $I_1$ is finite. Apply the lemma again to
$(a(n_k))$. Let $I_2$ be the resulting selected set, enlarged to contain
$n_0$. Then
\begin{equation}\label{eq:second-selection}
\sum_{n\in I_1}a(n)\le4\sum_{n\in I_2}a(n).
\end{equation}

Put
\[
\mathcal K:=\{k\ge1:n_k\in I_2\},\qquad
\mathcal K_s:=\{k\in\mathcal K:k\equiv s\pmod\kappa\},
\quad0\le s<\kappa.
\]
Let $k<\ell$ be consecutive elements of a non-empty $\mathcal K_s$ and set
\[
q:=\ell-k,\qquad d:=j(n_\ell)-j(n_k).
\]
Then $q\ge\kappa$. Since $n_i\in I_1$ for $i\ge1$,
\[
j(n_{i+1})\ge j(n_i)+1,
\]
and hence
\begin{equation}\label{eq:d-ge-q}
d\ge q\ge\kappa.
\end{equation}
The second selection gives
\[
a(n_\ell)<2^qa(n_k).
\]
Since $r=2^{\kappa-3}$,
\begin{equation}\label{eq:n-gap}
n_\ell-n_k\le q+(\kappa-3)d.
\end{equation}
Moreover,
\[
j(n_\ell-1)=j(n_\ell)=j(n_k)+d.
\]
From \eqref{eq:d-ge-q} and \eqref{eq:n-gap},
\[
n_\ell-1\le n_k+\kappa-1+\kappa(d-2).
\]
Lemma~\ref{lem:ball-calculus}(iii), with $d-2$ in place of $d$, yields
\begin{equation}\label{eq:ball-chain-inclusion}
4B_{j(n_\ell-1)}(t,2^{n_\ell-1})
\subset4B_{j(n_k)+2}(t,2^{n_k+\kappa-1}).
\end{equation}

Since $k\ge1$, $j(n_k-1)=j(n_k)$, and hence
$n_k-1\in M_{n_k-1}(t)$. Set
\[
\Phi_k:=F_{n_k-1}\bigl(4B_{j(n_k)}(t,2^{n_k-1})\cap T\bigr).
\]
By \eqref{eq:recursive-payment}, \eqref{eq:ball-chain-inclusion}, monotonicity
under inclusion, and $F_m\ge F_{m+1}$,
\begin{equation}\label{eq:telescoping-chain}
\frac1{4r}a(n_k)\le\Phi_k-\Phi_\ell.
\end{equation}
Summing \eqref{eq:telescoping-chain} separately on the $\kappa$ residue
classes gives
\[
\sum_{n_k\in I_2\setminus\{n_0\}}a(n_k)\lesssim_r F_0(T).
\]
For a finite chain, its last term is bounded by
$(4r)^{-1}a(n_k)\le\Phi_k\le F_0(T)$; for an infinite chain the differences
telescope. Finally,
\[
a(n_0)=a(1)=2r^{-j_0}.
\]
Combining this with \eqref{eq:first-selection} and
\eqref{eq:second-selection} proves the proposition.
\end{proof}

\subsection{Proof of Theorem~\ref{m-1}}

\begin{proof}[Proof of Theorem~\ref{m-1}]
Set
\[
\mathcal A_0=\mathcal A_1=\{T\},\qquad
q_0(T)=q_1(T)=j_0,
\]
and choose $t_{0,T}=t_{1,T}\in T$. By
\eqref{eq:initial-distance-compatibility},
\[
T\subset B_{j_0}(t_{0,T},1)\subset B_{j_0}(t_{1,T},2).
\]

Suppose that $n\ge1$ and that $\mathcal A_n$ has been constructed so that,
for every $A\in\mathcal A_n$,
\begin{equation}\label{eq:partition-induction}
j_n(t)=q_n(A)\quad(t\in A),\qquad
A\subset B_{q_n(A)}(t_{n,A},2^n).
\end{equation}
Fix $A\in\mathcal A_n$ and write $q=q_n(A)$. Put
\[
D_A:=\{t\in A:j_{n+1}(t)=q+1\},\qquad
C_A:=\{t\in A:j_{n+1}(t)=q\}.
\]

We claim that $D_A$ contains no points $t_1,\ldots,t_m$, with
$m=N_n/2$, such that
\begin{equation}\label{eq:separated-DA}
\varphi_{q+1}(t_\ell,t_{\ell'})>2^{n+1}\qquad(\ell\ne\ell').
\end{equation}
Otherwise set
\[
H_\ell:=4B_{q+2}(t_\ell,2^{n+\kappa})\cap T.
\]
The growth condition gives an index $h\le m$ such that
\[
F_n\left(\bigcup_{\ell=1}^mH_\ell\right)
\ge2^nr^{-q-1}+F_{n+1}(H_h).
\]
Since all $t_\ell$ belong to $B_q(t_{n,A},2^n)$,
Lemma~\ref{lem:ball-calculus}(i) gives
\[
\bigcup_{\ell=1}^mH_\ell\subset4B_q(t_h,2^n)\cap T.
\]
Hence
\[
F_n\bigl(4B_q(t_h,2^n)\cap T\bigr)
\ge2^nr^{-q-1}
+F_{n+1}\bigl(4B_{q+2}(t_h,2^{n+\kappa})\cap T\bigr).
\]
Since $t_h\in D_A$, we have $n\notin M_n(t_h)$. Hence, by the definition
of $M_n(t_h)$,
\[
F_n\bigl(4B_q(t_h,2^n)\cap T\bigr)
<
2^{n-1}r^{-q-1}
+
F_{n+1}\bigl(4B_{q+2}(t_h,2^{n+\kappa})\cap T\bigr),
\]
contradicting the preceding inequality.

Choose a maximal family $t_1,\ldots,t_L\in D_A$ satisfying
\eqref{eq:separated-DA}. Then $L<N_n/2$, and maximality gives
\[
D_A\subset\bigcup_{\ell=1}^L B_{q+1}(t_\ell,2^{n+1}).
\]
Define
\[
D_1:=D_A\cap B_{q+1}(t_1,2^{n+1}),
\]
and, for $2\le\ell\le L$,
\[
D_\ell:=D_A\cap\left(B_{q+1}(t_\ell,2^{n+1})
\setminus\bigcup_{h<\ell}B_{q+1}(t_h,2^{n+1})\right).
\]
After removing empty sets, give $D_\ell$ the label $q+1$ and center $t_\ell$.
If $C_A\ne\emptyset$, retain it as one further child, with label $q$ and
center $t_{n,A}$. Every child $A'$ of $A$ then satisfies
\[
q\le q_{n+1}(A')\le q+1,\qquad
A'\subset B_{q_{n+1}(A')}(t_{n+1,A'},2^{n+1}).
\]
Since $L+1\le N_n/2$, every $A\in\mathcal A_n$ has at most $N_n/2$
children. Therefore
\[
|\mathcal A_{n+1}|
\le |\mathcal A_n|\frac{N_n}{2}
\le N_{n+1}.
\]
This completes the recursion. By construction,
\[
q_n(A_n(t))=j_n(t),
\]
so Proposition~\ref{inequality-global} gives the asserted bound on the labels.

It remains to compare the labels with the partition functional. Define
\[
\mathcal B_0=\mathcal B_1=\mathcal B_2=\{T\},\qquad
\mathcal B_{n+2}:=\mathcal A_n\quad(n\ge1).
\]
This is an admissible sequence. If $A\in\mathcal A_n$ and $s,t\in A$, then
\[
\varphi_{q_n(A)}(s,t)
\le2\bigl(\varphi_{q_n(A)}(s,t_{n,A})
+\varphi_{q_n(A)}(t_{n,A},t)\bigr)
\le2^{n+2}.
\]
Hence
\[
j_{n+2}^\varphi(A)\ge q_n(A).
\]
Moreover, $\varphi_{j_0}\le1$ on $T$, so the contributions of
$\mathcal B_1$ and $\mathcal B_2$ are bounded by $6r^{-j_0}$. Thus
\[
\operatorname{Val}_\varphi(\mathcal B)
\lesssim r^{-j_0}
+\sup_{t\in T}\sum_{n\ge1}2^nr^{-q_n(A_n(t))}
\lesssim_r F_0(T)+r^{-j_0}.
\]
Taking the infimum over partition schemes proves the theorem.
\end{proof}

\subsection{The growth condition for points}

We now give the version for points used in Section~\ref{7}. Its proof is
independent of the preceding construction.

Let $(f_{n,j})_{n\ge0,\,j\ge j_0}$ be non-negative functions on $T$. We shall
use the monotonicity condition
\begin{equation}\label{eq:point-cone-monotonicity}
f_{n,j}(u)\ge f_{n',j'}(u)
\end{equation}
whenever
\[
j'>j\ge j_0,\qquad
0\le n'-n\le\kappa(j'-j),\qquad u\in T.
\]

\begin{definition}\label{groth con for point}
We say that $(f_{n,j})$ satisfies the \emph{growth condition for points} if,
for every $n\ge1$, $j\ge j_0$, $u\in T$, and every family
$t_1,\ldots,t_m\in T$, $m=N_n/2$, satisfying
\[
t_\ell\in2B_j(u,2^n)\qquad(1\le\ell\le m)
\]
and
\[
\varphi_{j+1}(t_\ell,t_{\ell'})>2^{n+1}
\qquad(\ell\ne\ell'),
\]
one has
\[
f_{n,j}(u)
\ge2^nr^{-j}+\min_{1\le\ell\le m}f_{n+\kappa,j+2}(t_\ell).
\]
\end{definition}

\begin{theorem}\label{Talagrand'}
Assume \eqref{eq:initial-distance-compatibility},
\eqref{eq:dj(s,t)=dj(t,s)}, \eqref{eq:dj<dj+dj}, and
\eqref{eq:ball-arithmetic}, and suppose that $j_0\in\mathbb Z$.
Let $(f_{n,j})$ satisfy \eqref{eq:point-cone-monotonicity} and the growth
condition for points. Put
\[
F_*:=\sup_{u\in T,\,n\ge0,\,j\ge j_0}f_{n,j}(u).
\]
Then
\[
\gamma_\varphi(T)\lesssim_r F_*+r^{-j_0}.
\]
More precisely, there exist a partition scheme $(\mathcal A_n)$, integers
$p_n(A)$, and centers $t_{n,A}\in T$ such that
\[
p_{n-1}(C)\le p_n(A)\le p_{n-1}(C)+1
\]
whenever $A\subset C$, with $A\in\mathcal A_n$ and
$C\in\mathcal A_{n-1}$,
\[
A\subset B_{p_n(A)}(t_{n,A},2^n),
\]
and
\[
\sup_{t\in T}\sum_{n\ge0}2^nr^{-p_n(A_n(t))}
\lesssim_r F_*+r^{-j_0}.
\]
\end{theorem}

Set
\[
p_0(t)=p_1(t)=j_0,\qquad t\in T.
\]
For $n\ge1$, once $p_n(t)$ has been defined, put
\begin{equation}\label{eq:point-scale-recursion}
p_{n+1}(t)=
\begin{cases}
p_n(t),&
 f_{n,p_n(t)}(t)\ge
 2^{n-1}r^{-p_n(t)}+f_{n+\kappa,p_n(t)+2}(t),\\
p_n(t)+1,&\text{otherwise}.
\end{cases}
\end{equation}
Thus $0\le p_{n+1}(t)-p_n(t)\le1$.

\begin{proposition}\label{inequality-points}
For the recursion \eqref{eq:point-scale-recursion},
\[
\sup_{t\in T}\sum_{n\ge0}2^nr^{-p_n(t)}
\lesssim_r F_*+r^{-j_0}.
\]
\end{proposition}

\begin{proof}
If $F_*=\infty$, there is nothing to prove. Fix $t\in T$ and write
\[
p(n):=p_n(t),\qquad a(n):=2^nr^{-p(n)}.
\]
The sequence $(a(n))_{n\ge1}$ is bounded. If
$p(n)=p(n-1)+1$, then
\[
a(n)=\frac2r a(n-1)\le\frac12a(n-1).
\]
If $p(n)=p(n-1)$ and $n\ge2$, the recursion at level $n-1$ gives
\[
\frac14a(n)\le f_{n-1,p(n)}(t)\le F_*.
\]

Apply Lemma~\ref{tala lem}, with $\alpha=2$, to $(a(n))_{n\ge1}$. Every
selected index $n>1$ satisfies
\[
p(n-1)=p(n),\qquad p(n+1)=p(n)+1.
\]
Hence, with
\[
I_1:=\{1\}\cup\{n>1:p(n-1)=p(n),\ p(n+1)=p(n)+1\},
\]
\begin{equation}\label{eq:point-first-selection}
\sum_{n\ge1}a(n)\le4\sum_{n\in I_1}a(n).
\end{equation}
Write
\[
I_1=\{n_0,n_1,\ldots\},\qquad1=n_0<n_1<\cdots.
\]
Apply Lemma~\ref{tala lem} again to $(a(n_k))_{k\ge0}$, and let $J$ be the
set provided by the lemma. Put
\[
I_2:=\{n_0\}\cup\{n_k:k\in J\}.
\]
Then
\begin{equation}\label{eq:point-second-selection}
\sum_{n\in I_1}a(n)\le4\sum_{n\in I_2}a(n).
\end{equation}

Set
\[
\mathcal K:=J\setminus\{0\},\qquad
\mathcal K_s:=\{k\in\mathcal K:k\equiv s\pmod\kappa\},
\quad0\le s<\kappa.
\]
Let $k<\ell$ be consecutive elements of a non-empty $\mathcal K_s$, and put
\[
q:=\ell-k,\qquad d:=p(n_\ell)-p(n_k).
\]
Then $q\ge\kappa$. Since $n_i\in I_1$ for $i\ge1$,
\[
p(n_{i+1})\ge p(n_i)+1,
\]
so
\begin{equation}\label{eq:point-d-ge-q}
d\ge q\ge\kappa.
\end{equation}
The second selection gives
\[
a(n_\ell)<2^qa(n_k),
\]
and hence
\begin{equation}\label{eq:point-n-gap}
n_\ell-n_k\le q+(\kappa-3)d.
\end{equation}
Since $n_\ell-n_k\ge\ell-k=q\ge\kappa$,
\[
0\le n_\ell-n_k-\kappa\le\kappa(d-2).
\]
By \eqref{eq:point-cone-monotonicity},
\begin{equation}\label{eq:point-future-comparison}
f_{n_k+\kappa-1,p(n_k)+2}(t)
\ge f_{n_\ell-1,p(n_\ell)}(t).
\end{equation}
Since $k\ge1$, $p(n_k-1)=p(n_k)$. The recursion at level $n_k-1$ gives
\[
\frac14a(n_k)
\le f_{n_k-1,p(n_k)}(t)
-f_{n_k+\kappa-1,p(n_k)+2}(t).
\]
Put
\[
\Phi_k:=f_{n_k-1,p(n_k)}(t).
\]
Together with \eqref{eq:point-future-comparison}, this yields
\begin{equation}\label{eq:point-telescoping}
\frac14a(n_k)\le\Phi_k-\Phi_\ell.
\end{equation}

For each $\ell=0,\ldots,\kappa-1$, sum
\eqref{eq:point-telescoping} over all $n$ such that
$n\equiv \ell \pmod{\kappa}$.
For an infinite chain the differences telescope because $\Phi_k\ge0$. For a
finite chain, its final term satisfies
\[
\frac14a(n_k)\le\Phi_k\le F_*.
\]
Thus
\[
\sum_{n_k\in I_2\setminus\{n_0\}}a(n_k)\lesssim_r F_*.
\]
Since $a(n_0)=a(1)=2r^{-j_0}$,
\eqref{eq:point-first-selection} and \eqref{eq:point-second-selection} give
\[
\sum_{n\ge1}a(n)\lesssim_r F_*+r^{-j_0}.
\]
Adding $a(0)=r^{-j_0}$ and taking the supremum over $t$ proves the
proposition.
\end{proof}

\begin{proof}[Proof of Theorem~\ref{Talagrand'}]
Set
\[
\mathcal A_0=\mathcal A_1=\{T\},\qquad
p_0(T)=p_1(T)=j_0,
\]
and choose $t_{0,T}=t_{1,T}\in T$. By
\eqref{eq:initial-distance-compatibility},
\[
T\subset B_{j_0}(t_{0,T},1)
\subset B_{j_0}(t_{1,T},2).
\]

Suppose that $n\ge1$ and that $\mathcal A_n$ has been constructed so that,
for every $A\in\mathcal A_n$,
\[
p_n(t)=p_n(A)\qquad(t\in A),
\]
and
\[
A\subset B_{p_n(A)}(t_{n,A},2^n).
\]
Fix $A\in\mathcal A_n$ and write
\[
q:=p_n(A).
\]
Put
\[
D_A:=\{t\in A:p_{n+1}(t)=q+1\},
\qquad
C_A:=\{t\in A:p_{n+1}(t)=q\}.
\]

We claim that $D_A$ contains no points
$t_1,\ldots,t_m$, where $m=N_n/2$, satisfying
\[
\varphi_{q+1}(t_\ell,t_{\ell'})>2^{n+1}
\qquad(\ell\ne\ell').
\]
Indeed, choose $h$ such that
\[
f_{n+\kappa,q+2}(t_h)
=
\min_{1\le\ell\le m}
f_{n+\kappa,q+2}(t_\ell).
\]
Since $t_\ell,t_h\in B_q(t_{n,A},2^n)$,
\eqref{eq:ball-arithmetic} gives
\[
t_\ell\in2B_q(t_h,2^n)
\qquad(1\le\ell\le m).
\]
The growth condition for points, applied with $u=t_h$, therefore gives
\[
f_{n,q}(t_h)
\ge
2^nr^{-q}
+
f_{n+\kappa,q+2}(t_h).
\]
On the other hand, $t_h\in D_A$, and hence, by
\eqref{eq:point-scale-recursion},
\[
f_{n,q}(t_h)
<
2^{n-1}r^{-q}
+
f_{n+\kappa,q+2}(t_h),
\]
a contradiction.

If $D_A\ne\emptyset$, choose a maximal family
$t_1,\ldots,t_L\in D_A$ such that
\[
\varphi_{q+1}(t_\ell,t_{\ell'})>2^{n+1}
\qquad(\ell\ne\ell').
\]
Then $L<N_n/2$, and maximality gives
\[
D_A
\subset
\bigcup_{\ell=1}^L
B_{q+1}(t_\ell,2^{n+1}).
\]
Define
\[
D_1
:=
D_A\cap B_{q+1}(t_1,2^{n+1}),
\]
and, for $2\le\ell\le L$,
\[
D_\ell
:=
D_A\cap
\left(
B_{q+1}(t_\ell,2^{n+1})
\setminus
\bigcup_{h<\ell}
B_{q+1}(t_h,2^{n+1})
\right).
\]
After removing the empty sets, put
\[
p_{n+1}(D_\ell)=q+1,
\qquad
t_{n+1,D_\ell}=t_\ell.
\]
If $C_A\ne\emptyset$, put
\[
p_{n+1}(C_A)=q,
\qquad
t_{n+1,C_A}=t_{n,A}.
\]
If $D_A=\emptyset$, only the latter set occurs.

Let $\mathcal A_{n+1}$ consist of all the non-empty sets obtained in this
way as $A$ ranges over $\mathcal A_n$. For every such set
$A'\subset A$,
\[
p_n(A)\le p_{n+1}(A')\le p_n(A)+1,
\]
and
\[
A'
\subset
B_{p_{n+1}(A')}(t_{n+1,A'},2^{n+1}).
\]
Moreover, each $A\in\mathcal A_n$ gives at most $L+1\le N_n/2$ sets in
$\mathcal A_{n+1}$. Hence
\[
|\mathcal A_{n+1}|
\le
|\mathcal A_n|\frac{N_n}{2}
\le
N_{n+1}.
\]
This completes the construction of the partition scheme. By construction,
\[
p_n(A_n(t))=p_n(t),
\]
and therefore Proposition~\ref{inequality-points} gives
\[
\sup_{t\in T}
\sum_{n\ge0}
2^nr^{-p_n(A_n(t))}
\lesssim_r
F_*+r^{-j_0}.
\]

It remains to compare these indices with the partition functional. Define
\[
\mathcal B_0=\mathcal B_1=\mathcal B_2=\{T\},
\qquad
\mathcal B_{n+2}:=\mathcal A_n
\quad(n\ge1).
\]
This is an admissible sequence. If $A\in\mathcal A_n$ and $s,t\in A$, then
\[
\begin{aligned}
\varphi_{p_n(A)}(s,t)
&\le
2\bigl(
\varphi_{p_n(A)}(s,t_{n,A})
+
\varphi_{p_n(A)}(t_{n,A},t)
\bigr) \\
&\le 2^{n+2}.
\end{aligned}
\]
Thus
\[
j_{n+2}^\varphi(A)\ge p_n(A).
\]
Moreover, $\varphi_{j_0}\le1$ on $T$, so the contributions of
$\mathcal B_1$ and $\mathcal B_2$ are bounded by $6r^{-j_0}$. Consequently,
\[
\operatorname{Val}_\varphi(\mathcal B)
\lesssim
r^{-j_0}
+
\sup_{t\in T}
\sum_{n\ge1}
2^nr^{-p_n(A_n(t))}
\lesssim_r
F_*+r^{-j_0}.
\]
Taking the infimum over partition schemes proves the theorem.
\end{proof}

\section{Parameterized separation trees and the proof of Theorem~\ref{m-2}}
\label{4}
We define the separation-tree functional and prove Theorem~\ref{m-2}.
The first inequality uses only
\eqref{eq:initial-distance-compatibility},
\eqref{eq:dj(s,t)=dj(t,s)}, and \eqref{eq:dj<dj+dj}.
The reverse inequality also uses
\eqref{eq:ball-arithmetic} and \eqref{mild condition}.

\begin{definition}\label{definition:abstract-tree}
Let $H$ be non-empty. A finite collection $\mathcal T$ of non-empty subsets
of $H$ is a \emph{tree on $H$} if $H\in\mathcal T$ and, whenever
$A,B\in\mathcal T$ intersect, either $A\subset B$ or $B\subset A$. The set
$H$ is the root. If $B\subsetneq A$ and no $C\in\mathcal T$ satisfies
$B\subsetneq C\subsetneq A$, then $B$ is a child of $A$. We write $c(A)$ for
the set of children of $A$ and
\[
 S_{\mathcal T}:=\bigcup\{A\in\mathcal T:c(A)=\emptyset\}.
\]
\end{definition}

\begin{definition}\label{tree}
Let $H\subset T$ be non-empty. A \emph{parameterized separation
$\varphi$-tree} on $H$ is a tree $\mathcal T$ on $H$ together with parameters
\[
 \mathbf n(A)\in\{1,2,\ldots\},
 \qquad
 \mathbf j(A)\in\mathbb Z\cup\{\infty\},
 \qquad A\in\mathcal T\setminus\{H\},
\]
satisfying the following conditions. If $A$ is not terminal, then all
$B\in c(A)$ have common parameters, say
\[
 \mathbf n(B)=n,
 \qquad
 \mathbf j(B)=j.
\]
If $A=H$, then $n\ge1$; if $A\ne H$, then $n>\mathbf n(A)$.

If $|A|\ge N_n/2$, then $j\in\mathbb Z$,
\[
 |c(A)|=N_n/2,
\]
and, for distinct $B,B'\in c(A)$,
\begin{equation}\label{eq:tree-separation}
 \varphi_{j+2}(B,B')>2^n,
\end{equation}
where
\[
 \varphi_j(B,B'):=\inf\{\varphi_j(s,t):s\in B,\ t\in B'\}.
\]
If $|A|<N_n/2$, then
\[
 c(A)=\{\{t\}:t\in A\},
 \qquad
 j=\infty.
\]
No condition is imposed on a terminal node.
\end{definition}

When $\varphi$ is fixed, we simply say separation tree.

\begin{definition}\label{def:rho(T)_a}
For a separation tree $\mathcal T$ on a non-empty set $H\subset T$, define
\[
\operatorname{Size}(\mathcal T)
:=r^{-j_0(H)}+
\inf_{t\in S_{\mathcal T}}
\sum_{\substack{A\in\mathcal T\setminus\{H\}\\t\in A}}
2^{\mathbf n(A)}r^{-\mathbf j(A)}.
\]
\end{definition}

Define
\begin{equation}\label{eq:rho-tree-functional}
\rho_\varphi(T)
:=\sup\{\operatorname{Size}(\mathcal T):
\mathcal T\text{ is a separation tree on a non-empty }H\subset T\}.
\end{equation}
The one-set tree on $T$ gives
\[
\rho_\varphi(T)\ge r^{-j_0}.
\]

We first prove the one-sided estimate.

\begin{theorem}\label{main theorem2}
Let $T\subset E$ be non-empty. Assume
\eqref{eq:initial-distance-compatibility},
\eqref{eq:dj(s,t)=dj(t,s)}, and \eqref{eq:dj<dj+dj}. Then
\[
 \rho_\varphi(T)\lesssim_r\mathcal M_\varphi(T).
\]
\end{theorem}

\begingroup
\begin{proof}
If $j_0=-\infty$, then
\[
\gamma_\varphi(T)=\mathcal M_\varphi(T)=\rho_\varphi(T)=\infty,
\]
and there is nothing to prove. If $j_0=\infty$, the three quantities vanish.
Thus we may assume that $j_0\in\mathbb Z$.

Fix a probability measure $\mu$ on $T$ with countable support, and let $\mathcal T$ be a separation tree on a non-empty set $H\subset T$. If $j_0(H)=\infty$, then \eqref{eq:initial-distance-compatibility} gives $\varphi_j(s,t)\le1$ for all $s,t\in H$ and all finite $j$. Hence no node of $\mathcal T$ with a finite scale parameter can satisfy the required separation. All non-root contributions are therefore zero, and $\operatorname{Size}(\mathcal T)=0$. We may assume that $j_0(H)<\infty$. We shall find $t_0\in S_{\mathcal T}$ such that
\[
\operatorname{Size}(\mathcal T)\lesssim_r I_{\mu,\varphi}(t_0).
\]

Let $B\in\mathcal T$ and suppose that $c(B)\ne\emptyset$. Let $n$ and $j$ be the common parameters of its children. We first assume that $n>2$ and $j<\infty$. For $A\in c(B)$, put
\[
q(A):=\sup_{t\in A}k_{n-2}^\mu(t).
\]
We claim that one child $A$ satisfies
\begin{equation}\label{eq:tree-measure-choice}
q(A)\le j+2.
\end{equation}
Suppose otherwise. Since $j<\infty$, the second case in Definition~\ref{tree} cannot occur. Hence
\[
c(B)=\{A_l:1\le l\le N_n/2\}.
\]
For each $l$, choose $t_l\in A_l$ and an integer $q_l\ge j+3$ such that
\[
\mu(B_{q_l}(t_l,2^{n-2}))\ge N_{n-2}^{-1}.
\]
This is possible because $q(A_l)>j+2$.

These balls are pairwise disjoint. Indeed, if $x$ belonged to the balls with indices $l$ and $l'$, then monotonicity and \eqref{eq:dj<dj+dj} would give
\begin{align*}
\varphi_{j+2}(t_l,t_{l'})
&\le2\bigl(\varphi_{j+2}(t_l,x)+\varphi_{j+2}(x,t_{l'})\bigr)\\
&\le2(2^{n-2}+2^{n-2})=2^n.
\end{align*}
This contradicts Definition~\ref{tree}. Therefore
\[
1=\mu(T)\ge\frac{N_n}{2N_{n-2}}>1,
\]
which is impossible. This proves \eqref{eq:tree-measure-choice}. If $j=\infty$, the children are singletons by Definition~\ref{tree}; we choose any one of them. Its contribution is zero.

Starting from $H$, repeat this choice whenever the common value of $\mathbf n$ is larger than $2$. At the other steps choose any child. The tree is finite, so the procedure ends at a set $D$ with $c(D)=\emptyset$. Choose $t_0\in D$.

Let $A$ be a set on the chosen branch and suppose that $A\ne H$, $\mathbf n(A)=n>2$, and $\mathbf j(A)<\infty$. Since $t_0\in A$, \eqref{eq:tree-measure-choice} gives
\[
k_{n-2}^\mu(t_0)\le\mathbf j(A)+2.
\]
Thus
\[
2^nr^{-\mathbf j(A)}
\le4r^2\,2^{n-2}r^{-k_{n-2}^\mu(t_0)}.
\]
The values of $\mathbf n$ increase along the branch, so each term on the right is used at most once.

It remains to consider $\mathbf n(A)\le2$. If $\mathbf j(A)<\infty$, then
\[
\mathbf j(A)\ge j_0(H)-1.
\]
Indeed, otherwise $\mathbf j(A)+2\le j_0(H)$. By \eqref{eq:initial-distance-compatibility} and monotonicity, two distinct sets at this step would satisfy
\[
\varphi_{\mathbf j(A)+2}\le\varphi_{j_0(H)}\le1,
\]
contrary to the required separation, which is larger than $2^{\mathbf n(A)}\ge2$. There are at most two such values of $\mathbf n(A)$. Their total contribution is at most a constant depending on $r$ times $r^{-j_0(H)}$.

Finally, $H\subset T$ gives $j_0(H)\ge j_0(T)$. Hence
\begin{align*}
\operatorname{Size}(\mathcal T)
&\lesssim_r r^{-j_0(T)}+\sum_{m\ge1}2^mr^{-k_m^\mu(t_0)}\\
&=I_{\mu,\varphi}(t_0).
\end{align*}
Taking the supremum over $H$ and $\mathcal T$, and then the infimum over $\mu$, proves the theorem.
\end{proof}
\endgroup

We next prove the reverse inequality under the additional ball assumptions.

\begin{theorem}\label{thm:upper_bound}
Let $T\subset E$ be non-empty. Assume
\eqref{eq:initial-distance-compatibility},
\eqref{eq:dj(s,t)=dj(t,s)}, \eqref{eq:dj<dj+dj},
\eqref{eq:ball-arithmetic}, and \eqref{mild condition}. Then
\[
 \mathcal M_\varphi(T)\lesssim_r\rho_\varphi(T).
\]
\end{theorem}

\begingroup
\begin{proof}
If $\rho_\varphi(T)=\infty$, there is nothing to prove. If $j_0=\infty$, then \eqref{eq:initial-distance-compatibility} gives $\varphi_j(s,t)\le1$ for every $s,t\in T$ and every finite $j$. Thus $\mathcal M_\varphi(T)=0$, and the conclusion again follows. We may therefore assume that $\rho_\varphi(T)<\infty$ and $j_0<\infty$.

For $n\ge0$ and every non-empty subset $U\subset T$, define
\[
F_n(U):=\sup_{H,\mathcal T}\ \inf_{t\in S_{\mathcal T}}
\sum_{\substack{A\in\mathcal T\setminus\{H\}\\t\in A}}
2^{\mathbf n(A)}r^{-\mathbf j(A)-1},
\]
where the supremum is taken over all non-empty $H\subset U$ and all separation trees $\mathcal T$ on $H$ such that
\[
\mathbf n(A)\ge n\qquad(A\in\mathcal T\setminus\{H\}).
\]
Set $F_n(\emptyset)=0$. Enlarging $U$ enlarges the collection of admissible trees, so $F_n$ is non-decreasing under inclusion. Increasing $n$ restricts this collection. Consequently,
\[
0\le F_{n+1}(U)\le F_n(U)\le r^{-1}\rho_\varphi(T)<\infty.
\]

We next verify the growth condition in Definition~\ref{groth con}. Fix $n\ge1$ and $j\in\mathbb Z$, and let $H_l$, $1\le l\le N_n/2$, satisfy the assumptions of that definition. Lemma~\ref{lem:ball-calculus}(ii) gives
\[
 \varphi_{j+2}(H_l,H_{l'})>2^n
 \qquad(l\ne l').
\]
Fix $\varepsilon>0$. For every $l$, choose a non-empty set $K_l\subset H_l$ and a separation tree $\mathcal T_l$ on $K_l$ such that all its sets different from $K_l$ have parameter at least $n+1$, and
\[
\inf_{t\in S_{\mathcal T_l}}
\sum_{\substack{A\in\mathcal T_l\setminus\{K_l\}\\t\in A}}
2^{\mathbf n(A)}r^{-\mathbf j(A)-1}
\ge F_{n+1}(H_l)-\varepsilon.
\]
Let
\[
K=\bigcup_{l\le N_n/2}K_l.
\]
Make the sets $K_l$ the children of $K$, and give them the parameters
\[
\mathbf n(K_l)=n,
\qquad
\mathbf j(K_l)=j.
\]
Let $\mathcal T$ be the tree obtained by taking $K$ as the root,
$K_l$ as its sons, and, below each $K_l$, the corresponding tree
$\mathcal T_l$. Since the sets $K_l$ are non-empty and pairwise separated,
they are pairwise disjoint. Moreover, $|K|\ge N_n/2$, so the root has at
least $N_n/2$ sons. These sons have parameters $(n,j)$ and satisfy the
required separation condition. The parameters of their descendants are at
least $n+1$. Hence $\mathcal T$ is a separation tree rooted at $K$.

It follows that
\[
F_n\Bigl(\bigcup_{l\le N_n/2}H_l\Bigr)
\ge2^nr^{-j-1}+\min_{l\le N_n/2}F_{n+1}(H_l)-\varepsilon.
\]
Letting $\varepsilon\to0$, we obtain the growth condition.

Since $j_0<\infty$, condition~\eqref{eq:initial-distance-compatibility} gives $\varphi_{j_0}(s,t)\le1$ for all $s,t\in T$. Applying Theorem~\ref{m-1} to the functionals $(F_n)$ therefore gives
\[
\gamma_\varphi(T)\lesssim_r rF_0(T)+r^{-j_0}.
\]
By the definition of $F_0$ and \eqref{eq:rho-tree-functional},
\[
rF_0(T)\le\rho_\varphi(T),
\qquad
r^{-j_0}\le\rho_\varphi(T).
\]
Consequently,
\[
\gamma_\varphi(T)\lesssim_r\rho_\varphi(T).
\]
Using Proposition~\ref{equivalence}, we obtain
\[
\mathcal M_\varphi(T)\lesssim_r\gamma_\varphi(T)
\lesssim_r\rho_\varphi(T),
\]
as required.
\end{proof}
\endgroup

Combining the preceding estimates gives the following equivalent form of Theorem~\ref{m-2}.

\begin{theorem}\label{thm:equv}
Let $T\subset E$ be non-empty. Assume
\eqref{eq:initial-distance-compatibility},
\eqref{eq:dj(s,t)=dj(t,s)}, and \eqref{eq:dj<dj+dj}. Then
\[
 \rho_\varphi(T)\lesssim_r\mathcal M_\varphi(T).
\]
If, in addition, \eqref{eq:ball-arithmetic} and
\eqref{mild condition} hold, then
\[
 \rho_\varphi(T)\asymp_r\mathcal M_\varphi(T)
 \asymp_r\gamma_\varphi(T).
\]
\end{theorem}

\begin{proof}
The first assertion is Theorem~\ref{main theorem2}. Under the additional assumptions \eqref{eq:ball-arithmetic} and
\eqref{mild condition}, Theorem~\ref{thm:upper_bound} gives
\[
\mathcal M_\varphi(T)\lesssim_r\rho_\varphi(T).
\]
Together with the first assertion, this yields
\[
\rho_\varphi(T)\asymp_r\mathcal M_\varphi(T).
\]
Finally, Proposition~\ref{equivalence} gives
\[
\mathcal M_\varphi(T)\asymp_r\gamma_\varphi(T),
\]
which proves the second assertion and completes the proof of Theorem~\ref{m-2}.
\end{proof}
    
\section{Examples of Theorem~\ref{m-2}}
\label{5}
We first consider canonical processes with regular log-concave tails and then
turn to Bernoulli processes.

\subsection{Canonical processes with regular log-concave tails}

Let $(X_i)_{i\ge1}$ be independent symmetric random variables and put
\[
 U_i(s):=-\log \textsf P(|X_i|\ge s),\qquad s\ge0.
\]
Assume that each $U_i$ is convex and, after rescaling the coordinates, that
\[
 U_i(1)=1.
\]
We also assume Lata{\l}a's regularity conditions: for some $C_0\ge1$,
\begin{equation}\label{eq:regular-tail-assumptions}
 U_i(2s)\le C_0U_i(s),\qquad s\ge1,
 \qquad
 U_i'(0+)\ge C_0^{-1}.
\end{equation}
Set
\[
 \widehat U_i(x):=
 \begin{cases}
  x^2,& |x|\le1,\\
  2U_i(|x|)-1,& |x|>1.
 \end{cases}
\]
The function $\widehat U_i$ is even and convex. At $1$ the two pieces have
the same value, and convexity of $U_i$, together with $U_i(0)=0$ and
$U_i(1)=1$, gives $U_i'(1+)\ge1$, so the right derivative of the second
piece is at least the left derivative of $x^2$.

For $u>0$ define
\[
 C(u):=
 \left\{a\in c_{00}:\sum_{i\ge1}\widehat U_i(a_i)\le u\right\},
\]
\[
 N_u(t):=
 \sup_{a\in C(u)}\left|\sum_{i\ge1}t_i a_i\right|,
 \qquad
 B(u):=\{t\in\ell^2:N_u(t)\le u\}.
\]
The sets $B(u)$ are symmetric, convex, and increasing in $u$. Indeed, if
$0<u\le v$, convexity of the functions $\widehat U_i$ gives
\[
 C(v)\subset \frac vu C(u),
\]
and therefore
\[
 N_v(t)\le\frac vu N_u(t).
\]
In particular, $B(u)\subset B(v)$.

Fix $\kappa\ge5$, to be chosen below, put $r=2^{\kappa-3}$, and define
\begin{equation}\label{eq:regular-logconcave-family}
 \varphi_j(s,t)
 :=\inf\{u>0:r^j(s-t)\in B(u)\},
 \qquad j\in\mathbb Z.
\end{equation}
This is the standard family used in Lata{\l}a's partition theorem. It satisfies
\eqref{eq:dj(s,t)=dj(t,s)} and \eqref{eq:dj<dj+dj}; see
\cite{lata1,talab}. Since the sets $B(u)$ are closed and increasing,
\begin{equation}\label{eq:regular-ball-identification}
 B_j(c)=r^{-j}B(c),
 \qquad c>0.
\end{equation}

For the initial distance put
\[
 d_0(s,t):=\inf\{a>0:s-t\in aB(1)\}.
\]
By \eqref{eq:regular-ball-identification},
\[
 d_0(s,t)\le r^{-j}
 \quad\Longleftrightarrow\quad
 \varphi_j(s,t)\le1.
\]
Thus \eqref{eq:initial-distance-compatibility} holds. Moreover,
\[
j_0
=
\sup\bigl\{
j\in\mathbb Z:T-T\subset r^{-j}B(1)
\bigr\}.
\]

We record the two ball relations used in the reverse inequality of
Theorem~\ref{m-2}.
Since $B(c)$ is convex and contains the origin,
\[
 aB(c)+bB(c)\subset(a+b)B(c),
 \qquad a,b\ge0.
\]
Together with \eqref{eq:regular-ball-identification}, this proves
\eqref{eq:ball-arithmetic}.

The assumptions \eqref{eq:regular-tail-assumptions} imply
\begin{equation}\label{eq:modified-tail-doubling}
 \widehat U_i(2x)\le K\widehat U_i(x),
 \qquad x\in\mathbb R,
\end{equation}
uniformly in $i$, where one may take $K=8C_0$. For $|x|\le1/2$ this follows
from the quadratic part. If $1/2<|x|\le1$, then
$U_i(2|x|)\le U_i(2)\le C_0$, while $x^2\ge1/4$. If $|x|>1$, then
$U_i(|x|)\ge1$ and
\[
 2U_i(2|x|)-1
 \le2C_0U_i(|x|)
 \le2C_0\widehat U_i(x).
\]
Choose $\kappa$ so large that
\[
 r=2^{\kappa-3}\ge K^8.
\]
We claim that
\begin{equation}\label{eq:regular-ball-contraction}
 B(8ru)\subset \frac r{20}B(u),
 \qquad u>0.
\end{equation}
Let $t\in B(8ru)$ and let $(b_i)\in c_{00}$ satisfy
\[
 \sum_i\widehat U_i(b_i)\le u.
\]
Iterating \eqref{eq:modified-tail-doubling} gives
\[
 \sum_i\widehat U_i(2^8b_i)
 \le K^8u
 \le ru
 \le8ru.
\]
Since $N_{8ru}(t)\le8ru$,
\[
 2^8\left|\sum_i b_it_i\right|\le8ru.
\]
Hence
\[
 \left|\sum_i b_i\frac{20t_i}{r}\right|
 \le\frac{160}{2^8}u
 \le u.
\]
Taking the supremum over all such $(b_i)$ yields
\[
 N_u\left(\frac{20t}{r}\right)\le u,
\]
which proves \eqref{eq:regular-ball-contraction}. Since $2^\kappa=8r$,
\eqref{eq:regular-ball-identification} now gives
\[
 \begin{aligned}
 B_{j+1}(2^{n+\kappa})
 &=r^{-j-1}B(8r2^n)\\
 &\subset \frac1{20}r^{-j}B(2^n)
 =\frac1{20}B_j(2^n).
 \end{aligned}
\]
Thus \eqref{mild condition} holds.

Lata{\l}a's majorizing-measure theorem, in its partition form, gives
\begin{equation}\label{eq:latala-partition-comparison}
 \textsf E\sup_{t\in T}X_t
 \asymp_{C_0,r}\gamma_\varphi(T).
\end{equation}
Here, as throughout the paper, the assertion for an arbitrary $T$ is
understood through its finite subsets.

Combining \eqref{eq:latala-partition-comparison} with Theorem~\ref{m-2} gives
the following form of the tree representation.

\begin{theorem}\label{canonical processes}
Under the assumptions above,
\[
 \textsf E\sup_{t\in T}X_t
 \asymp_{C_0,r}\rho_\varphi(T).
\]
\end{theorem}

\subsection{A one-sided estimate for Bernoulli processes}
For Bernoulli processes we use only the first assertion of Theorem~\ref{m-2}. Together with Theorem~\ref{bernoulli partition} and Proposition~\ref{equivalence}, it gives
\begin{proposition}
    Let $(\varphi_j)_{j\in\mathbb Z}$ be the family of distances for Bernoulli processes defined in \eqref{bernoulli: family of distances}, with $d_0(s,t)=\|s-t\|_2$. Then
    $$
    \rho_\varphi(T)\lesssim_r \textsf{E}\sup\limits_{t\in T}\mathcal{E}_t.$$
\end{proposition}
The reverse inequality is left as a conjecture for general $T$. Section~\ref{6} proves it for bounded convex unconditional sets.
\begin{conjecture}\label{conjecture}
    Let $(\varphi_j)_{j\in\mathbb Z}$ be the family of distances for Bernoulli processes defined in \eqref{bernoulli: family of distances}, with $d_0(s,t)=\|s-t\|_2$. Then
    $$
    \textsf{E}\sup\limits_{t\in T}\mathcal{E}_t\lesssim_r\rho_\varphi(T).$$
\end{conjecture}

\begingroup
\section{The Bernoulli case for bounded convex unconditional sets}
\label{6}

We prove Conjecture~\ref{conjecture} for bounded convex unconditional subsets of $\mathbb R^d$. The proof first treats coordinate rectangles.

In this section, $(\varphi_j)_{j\in\mathbb Z}$ is given by \eqref{bernoulli: family of distances}, and
\[
d_0(s,t)=\|s-t\|_2.
\]
Throughout this section, $j_0$ is the initial scale determined by $d_0$.
For $1\le m\le d$ and $a_1,\ldots,a_m\ge0$, let
\begin{equation}\label{eq:coordinate-rectangle}
Q(a)=\left\{x\in\mathbb R^d:|x_i|\le a_i\text{ for }1\le i\le m,
\quad x_i=0\text{ for }i>m\right\}.
\end{equation}

\begin{proposition}\label{prop:coordinate-rectangle}
Let $T\subset\mathbb R^d$. If $Q(a)\subset T$, then
\[
\sum_{i=1}^m a_i\lesssim_r\rho_\varphi(T).
\]
\end{proposition}

\begin{proof}
We may omit the zero terms and reorder the coordinates. Thus we assume
\[
a_1\ge a_2\ge\cdots\ge a_m>0.
\]
Put
\[
x=(a_1,\ldots,a_m,0,\ldots,0).
\]
Since $x,-x\in T$,
\begin{equation}\label{eq:coordinate-rectangle-l2}
2\|a\|_2=\|x-(-x)\|_2\le\operatorname{diam}_{d_0}(T)\le r^{-j_0}.
\end{equation}
The one-set tree on $T$ gives
\begin{equation}\label{eq:root-below-rho}
r^{-j_0}\le\rho_\varphi(T).
\end{equation}

Set $L=256$. For $n\ge1$, let
\[
m_n=L2^n,
\qquad
M_n=\sum_{k=1}^n m_k,
\qquad
M_0=0.
\]
Let $q$ be the largest integer such that $M_q\le m$. If no such integer exists, then $m<M_1$ and
\[
\sum_{i=1}^m a_i\le\sqrt m\,\|a\|_2
\le\sqrt{M_1}\,\|a\|_2
\lesssim r^{-j_0}\le\rho_\varphi(T).
\]
We may therefore assume that $q\ge1$.

For $1\le n\le q$, set
\[
I_n=\{M_{n-1}+1,\ldots,M_n\}.
\]
We next choose a set $\mathcal C_n\subset\{-1,1\}^{I_n}$ with $|\mathcal C_n|=N_n/2$ such that two distinct elements differ in more than $m_n/16$ coordinates.

To see that this is possible, choose sign vectors one at a time. After choosing a vector, remove all vectors which differ from it in at most $m_n/16$ coordinates. At each step, at most
\[
\sum_{k=0}^{m_n/16}\binom{m_n}{k}
\le(m_n+1)\binom{m_n}{m_n/16}
\le(m_n+1)(16e)^{m_n/16}
\]
vectors are removed. Since $m_n\ge512$,
\[
(m_n+1)(16e)^{m_n/16}<2^{m_n/2}.
\]
When the procedure stops, every sign vector has been chosen or removed. Since there are $2^{m_n}$ sign vectors, the number chosen is at least
\[
\frac{2^{m_n}}{2^{m_n/2}}=2^{m_n/2}.
\]
Also,
\[
2^{m_n/2}\ge2^{2^n-1}=N_n/2.
\]
Keeping $N_n/2$ vectors gives $\mathcal C_n$.

For
\[
\omega=(\omega_1,\ldots,\omega_q)
\in\prod_{n=1}^q\mathcal C_n,
\]
define $x^\omega\in Q(a)$ by
\[
x_i^\omega=
\begin{cases}
\omega_n(i)a_i,&i\in I_n,\ 1\le n\le q,\\
0,&i>M_q.
\end{cases}
\]
For $1\le n\le q$, let
\[
A(\omega_1,\ldots,\omega_n)
=\{x^\eta:\eta_k=\omega_k\text{ for }1\le k\le n\}.
\]
We now construct a tree with root $T$. The children of the root are the sets $A(\omega_1)$,
$\omega_1\in\mathcal C_1$. Thus the root has $N_1/2$ children. For $2\le n\le q$, the children of $A(\omega_1,\ldots,\omega_{n-1})$ are the sets $A(\omega_1,\ldots,\omega_n)$, where $\omega_n\in\mathcal C_n$. Such a set at level $n-1$ contains
\[
\prod_{k=n}^q|\mathcal C_k|\ge|\mathcal C_n|=N_n/2
\]
points, and it has exactly $N_n/2$ children. The sets at level $q$ are singletons.

For every $1\le n\le q$, choose an integer $j_n$ such that
\begin{equation}\label{eq:special-jn-choice}
a_{M_n}\le r^{-j_n}<ra_{M_n}.
\end{equation}
Give every set at level $n$ the parameters
\[
\mathbf n=n,
\qquad
\mathbf j=j_n.
\]
Let $A$ and $A'$ be distinct children at level $n$. Their sign vectors differ in more than $m_n/16$ coordinates. If $i$ is one of these coordinates, then for every $y\in A$ and $z\in A'$,
\[
|y_i-z_i|=2a_i\ge2a_{M_n}.
\]
By \eqref{eq:special-jn-choice},
\[
r^{j_n+2}|y_i-z_i|>2r>1.
\]
Thus each such coordinate contributes $1$ to $\varphi_{j_n+2}(y,z)$. It follows that
\[
\varphi_{j_n+2}(A,A')>\frac{m_n}{16}=16\,2^n>2^n.
\]
Hence this is a parameterized separation tree. Every branch contains one set at each level $1,\ldots,q$. Therefore
\begin{equation}\label{eq:special-tree-sum}
\rho_\varphi(T)
\ge r^{-j_0}+\sum_{n=1}^q2^nr^{-j_n}
\ge r^{-j_0}+\sum_{n=1}^q2^na_{M_n}.
\end{equation}

It remains to compare this sum with $\sum_i a_i$. First,
\[
\sum_{i=1}^{M_1}a_i\le\sqrt{M_1}\,\|a\|_2\lesssim r^{-j_0}.
\]
For $2\le n\le q$,
\[
\sum_{i=M_{n-1}+1}^{M_n}a_i
\le m_na_{M_{n-1}}
=2L\,2^{n-1}a_{M_{n-1}}.
\]
Since $M_{q+1}>m$, the last part has fewer than $m_{q+1}$ terms. Hence
\[
\sum_{i=M_q+1}^{m}a_i
\le m_{q+1}a_{M_q}
=2L\,2^qa_{M_q}.
\]
Consequently,
\[
\sum_{i=1}^m a_i
\lesssim r^{-j_0}+\sum_{n=1}^q2^na_{M_n}
\lesssim_r\rho_\varphi(T),
\]
where the last inequality follows from \eqref{eq:special-tree-sum}.
\end{proof}

Recall that $T\subset\mathbb R^d$ is unconditional if
\[
(t_1,\ldots,t_d)\in T
\quad\Longrightarrow\quad
(\theta_1t_1,\ldots,\theta_dt_d)\in T
\]
for every $\theta_1,\ldots,\theta_d\in\{-1,1\}$.

We next pass from rectangles to convex unconditional sets.

\begin{theorem}\label{thm:unconditional-convex-set}
Let $T\subset\mathbb R^d$ be bounded, convex, and unconditional. Then
\[
\textsf E\sup_{t\in T}\mathcal E_t\asymp_r\rho_\varphi(T).
\]
\end{theorem}

\begin{proof}
Fix $\varepsilon=(\varepsilon_1,\ldots,\varepsilon_d)\in\{-1,1\}^d$. Since $T$ is unconditional,
\begin{equation}\label{eq:unconditional-width}
\sup_{t\in T}\sum_{i=1}^d\varepsilon_it_i
=\sup_{t\in T}\sum_{i=1}^d|t_i|.
\end{equation}
Indeed, for every $t\in T$, the point $(\varepsilon_i|t_i|)_{i\le d}$ also belongs to $T$.

Fix $t\in T$. All points
\[
(\theta_1t_1,\ldots,\theta_dt_d),
\qquad \theta_i\in\{-1,1\},
\]
belong to $T$. Their convex hull is $Q(|t_1|,\ldots,|t_d|)$. Since $T$ is convex, Proposition~\ref{prop:coordinate-rectangle} gives
\[
\sum_{i=1}^d|t_i|\lesssim_r\rho_\varphi(T).
\]
Taking the supremum over $t$ in \eqref{eq:unconditional-width}, and then the expectation, gives
\[
\textsf E\sup_{t\in T}\mathcal E_t\lesssim_r\rho_\varphi(T).
\]
The reverse inequality follows from Theorem~\ref{main theorem2}, Proposition~\ref{equivalence}, and Theorem~\ref{bernoulli partition}.
\end{proof}

For coordinate ellipsoids this gives an explicit formula.

\begin{corollary}\label{cor:coordinate-ellipsoid}
Let $d\ge1$, let $a_1,\ldots,a_d>0$, and set
\[
E_a=\left\{t\in\mathbb R^d:\sum_{i=1}^d\frac{t_i^2}{a_i^2}\le1\right\}.
\]
Then
\[
\rho_\varphi(E_a)\asymp_r\left(\sum_{i=1}^da_i^2\right)^{1/2}.
\]
\end{corollary}

\begin{proof}
The set $E_a$ is bounded, convex, and unconditional. For every $\varepsilon\in\{-1,1\}^d$, the Cauchy--Schwarz inequality gives
\[
\sup_{t\in E_a}\sum_{i=1}^d\varepsilon_it_i
\le\left(\sum_{i=1}^da_i^2\right)^{1/2}.
\]
Equality holds for
\[
t_i=\frac{a_i^2\varepsilon_i}{(\sum_{k=1}^da_k^2)^{1/2}}.
\]
Therefore
\[
\textsf E\sup_{t\in E_a}\mathcal E_t
=\left(\sum_{i=1}^da_i^2\right)^{1/2}.
\]
The conclusion follows from Theorem~\ref{thm:unconditional-convex-set}.
\end{proof}

\begin{remark}
The preceding corollary concerns ellipsoids whose principal axes are the coordinate axes. The argument does not apply to a general orthogonal image of such an ellipsoid. The distances in \eqref{bernoulli: family of distances} depend on the coordinates. They are not invariant under orthogonal transformations. Also, Theorem~\ref{thm:unconditional-convex-set} uses unconditionality, not only central symmetry.
\end{remark}
\endgroup

\section{The finite-set construction}\label{7}

Throughout this section $T$ is finite and non-empty, $j_0\in\mathbb Z$, and
the assumptions in the second assertion of Theorem~\ref{m-2} are satisfied.
Put
\[
 n_T:=\max\{n\ge1:N_n/2\le |T|\},
\]
with $n_T=0$ if the set on the right is empty.

Since $\varphi_{j+1}\ge\varphi_j$, the sets $2B_j(x,2^n)$ decrease with $j$.
Put
\[
\mathcal Q
:=
\left\{
q\in\mathbb Z:
\begin{array}{l}
\text{there exist }x,y\in T,\ 1\le n\le n_T\text{ such that}\\[1mm]
y\in2B_q(x,2^n),\qquad y\notin2B_{q+1}(x,2^n)
\end{array}
\right\}.
\]
For each fixed $x,y,n$, there is at most one such integer $q$, and hence
\[
 |\mathcal Q|\le |T|^2n_T.
\]
Set
\[
\mathcal J
:=
\{j_0\}
\cup
\bigcup_{q\in\mathcal Q}
\{q-2,q-1,q,q+1,q+2\}.
\]
Then
\begin{equation}\label{eq:active-scale-count}
 |\mathcal J|\le1+5|T|^2n_T.
\end{equation}
Finally, put
\[
 J_+:=\max\mathcal J+2.
\]

\subsection{A finite recursion}

For $n\ge1$, $j\in\mathbb Z$, and $u,x\in T$, put
\[
 D_{n,j}(u):=2B_j(u,2^n)\cap T,
 \qquad
 Q_{n,j}(x):=9B_{j+2}(x,2^{n+\kappa})\cap T.
\]
Assume that $f_{n+\kappa,j+2}$ has already been defined, and set
$m=N_n/2$. Choose $x_1\in D_{n,j}(u)$ so that
\[
 f_{n+\kappa,j+2}(x_1)
 =\max_{x\in D_{n,j}(u)}f_{n+\kappa,j+2}(x).
\]
Suppose that $x_1,\ldots,x_\ell$ have been chosen, where $1\le\ell<m$.
If
\[
 D_{n,j}(u)\subset\bigcup_{i=1}^{\ell}Q_{n,j}(x_i),
\]
stop and put $A_{n,j}(u)=0$. Otherwise choose
\[
 x_{\ell+1}\in
 D_{n,j}(u)\setminus\bigcup_{i=1}^{\ell}Q_{n,j}(x_i)
\]
so that $f_{n+\kappa,j+2}(x_{\ell+1})$ is maximal on this set. If the
construction reaches $m$ points, define
\begin{equation}\label{eq:algorithm-packing-value}
 A_{n,j}(u):=2^nr^{-j}+f_{n+\kappa,j+2}(x_m).
\end{equation}
For $n=0$ or $n>n_T$, set $A_{n,j}=0$.

For pairs $(n,j)$ write
\begin{equation}\label{eq:algorithm-order}
 (n,j)\prec(n',j')
 \quad\Longleftrightarrow\quad
 j'>j,
 \qquad
 0\le n'-n\le\kappa(j'-j).
\end{equation}
Set
\[
 f_{n,j}=0
 \qquad\text{if }n>n_T\text{ or }j>J_+.
\]
For $0\le n\le n_T$ and $j_0\le j\le J_+$, define the remaining functions,
in decreasing order of $j$, by
\begin{equation}\label{eq:algorithm-recursion}
 f_{n,j}(u)
 :=\max\left\{
 A_{n,j}(u),
 \max_{\substack{(n,j)\prec(n',j')\\
                  0\le n'\le n_T,\ j<j'\le J_+}}
 f_{n',j'}(u)
 \right\},
\end{equation}
where the second maximum is understood to be zero if the set of indices is
empty. Thus the recursion is finite.

\begin{proposition}\label{prop:algorithm-growth}
The functions defined by \eqref{eq:algorithm-recursion} satisfy
\eqref{eq:point-cone-monotonicity} and the growth condition for points in
Definition~\ref{groth con for point}.
\end{proposition}

\begin{proof}
Condition~\eqref{eq:point-cone-monotonicity} follows directly from
\eqref{eq:algorithm-recursion}. Let $n\ge1$, $m=N_n/2$, and suppose that
$t_1,\ldots,t_m\in2B_j(u,2^n)$ satisfy
\[
 \varphi_{j+1}(t_\ell,t_{\ell'})>2^{n+1}
 \qquad(\ell\ne\ell').
\]
The points are distinct, so $n\le n_T$. We also have $j\le J_+$. Indeed,
Lemma~\ref{lem:ball-calculus}(v) gives
\[
 t_2\notin Q_{n,j}(t_1).
\]
Since $t_1,t_2\in2B_j(u,2^n)$, they cannot both belong to
$2B_{j+2}(u,2^n)$; otherwise
\[
 t_1-t_2\in4B_{j+2}(2^n)
 \subset9B_{j+2}(2^{n+\kappa}),
\]
contrary to the preceding display. Hence, for one of the two points, say
$t_\ell$,
\[
 t_\ell\in2B_j(u,2^n),
 \qquad
 t_\ell\notin2B_{j+2}(u,2^n).
\]
Since the sets $2B_q(u,2^n)$ decrease with $q$, either $j\in\mathcal Q$ or
$j+1\in\mathcal Q$. Thus $j\in\mathcal J$, and in particular $j<J_+$.

By Lemma~\ref{lem:ball-calculus}(v), every set $Q_{n,j}(x)$ contains at most
one of the points $t_1,\ldots,t_m$. Hence the construction yields $m$ points
$x_1,\ldots,x_m$. Before $x_m$ is chosen, at least one of the points
$t_\ell$ belongs to
\[
 D_{n,j}(u)\setminus\bigcup_{i<m}Q_{n,j}(x_i).
\]
By the choice of $x_m$,
\[
 f_{n+\kappa,j+2}(x_m)
 \ge\min_{1\le\ell\le m}f_{n+\kappa,j+2}(t_\ell).
\]
Equations~\eqref{eq:algorithm-packing-value} and
\eqref{eq:algorithm-recursion} therefore give
\[
 f_{n,j}(u)
 \ge2^nr^{-j}
 +\min_{1\le\ell\le m}f_{n+\kappa,j+2}(t_\ell),
\]
which is the required inequality.
\end{proof}

\begin{theorem}\label{thm:algorithm-upper}
For the functions in \eqref{eq:algorithm-recursion},
\[
 \gamma_\varphi(T)
 \lesssim_r r^{-j_0}
 +\sup_{u\in T,\,n\ge0,\,j\ge j_0}f_{n,j}(u).
\]
\end{theorem}

\begin{proof}
This follows from Theorem~\ref{Talagrand'} and
Proposition~\ref{prop:algorithm-growth}.
\end{proof}

\subsection{Construction of a separation tree}

We next prove the reverse estimate.

\begin{theorem}\label{thm:algorithm-tree-certificate}
For the functions in \eqref{eq:algorithm-recursion},
\[
 r^{-j_0}
 +\sup_{u\in T,\,n\ge0,\,j\ge j_0}f_{n,j}(u)
 \le\rho_\varphi(T).
\]
Consequently,
\[
 r^{-j_0}+\sup_{u,n,j}f_{n,j}(u)
 \asymp_r\gamma_\varphi(T).
\]
\end{theorem}

\begin{proof}
Put
\[
 F_*:=\sup_{u\in T,\,n\ge0,\,j\ge j_0}f_{n,j}(u).
\]
If $F_*=0$, the one-set tree on $T$ gives the first assertion. Assume
$F_*>0$. Only finitely many of the functions $f_{n,j}$ are non-zero, so
there exist $n^*,j^*$ and $u^*\in T$ such that
\[
 F_*=f_{n^*,j^*}(u^*).
\]
If the second term in \eqref{eq:algorithm-recursion} gives this value, choose
$(n',j')$ with $(n^*,j^*)\prec(n',j')$ and
\[
 f_{n^*,j^*}(u^*)=f_{n',j'}(u^*),
\]
and replace $(n^*,j^*)$ by $(n',j')$. Since $j'$ strictly increases, after
finitely many steps we obtain, with the same notation,
\begin{equation}\label{eq:algorithm-active-max}
 f_{n^*,j^*}(u^*)=A_{n^*,j^*}(u^*)>0.
\end{equation}
In particular $n^*\ge1$.

Run the construction defining $A_{n^*,j^*}(u^*)$ and put
$m=N_{n^*}/2$. It gives points $x_1,\ldots,x_m$ such that
\[
 x_q\notin Q_{n^*,j^*}(x_i)
 \qquad(i<q).
\]
By the choice of the $x_i$,
\[
 f_{n^*+\kappa,j^*+2}(x_1)
 \ge\cdots\ge
 f_{n^*+\kappa,j^*+2}(x_m),
\]
and hence
\begin{equation}\label{eq:algorithm-split-identity}
 f_{n^*,j^*}(u^*)
 =2^{n^*}r^{-j^*}
 +\min_{1\le i\le m}f_{n^*+\kappa,j^*+2}(x_i).
\end{equation}
For $1\le i\le m$, put
\[
 A_i:=4B_{j^*+2}(x_i,2^{n^*+\kappa})\cap T,
 \qquad
 \mathbf n(A_i):=n^*,
 \qquad
 \mathbf j(A_i):=j^*.
\]
Each $A_i$ is non-empty, and Lemma~\ref{lem:ball-calculus}(iv) gives
\[
 \varphi_{j^*+2}(A_i,A_{i'})>2^{n^*}
 \qquad(i\ne i').
\]
Thus the sets $A_i$ are pairwise disjoint. Since there are $N_{n^*}/2$ of
them, take $T$ as the root and the sets $A_i$ as its children. Put
\[
 g(T):=F_*,
 \qquad
 g(A_i):=f_{n^*+\kappa,j^*+2}(x_i).
\]

We continue below each $A_i$. If $g(A_i)=0$, leave $A_i$ terminal. Otherwise,
starting with $(n^*+\kappa,j^*+2)$, repeat the preceding choice of later pairs
until indices $(n',j')$ are obtained such that
\[
 f_{n^*+\kappa,j^*+2}(x_i)
 =f_{n',j'}(x_i)
 =A_{n',j'}(x_i)>0.
\]
By transitivity of \eqref{eq:algorithm-order},
\begin{equation}\label{eq:algorithm-descendant-cone}
 j'\ge j^*+2,
 \qquad
 0\le n'-(n^*+\kappa)
 \le\kappa\bigl(j'-(j^*+2)\bigr).
\end{equation}
Run the construction defining $A_{n',j'}(x_i)$, with centers
$x'_1,\ldots,x'_{N_{n'}/2}$, and put
\[
 A'_q:=4B_{j'+2}(x'_q,2^{n'+\kappa})\cap T.
\]
Since $x'_q\in2B_{j'}(x_i,2^{n'})$, Lemma~\ref{lem:ball-calculus}(vi) gives
\[
 A'_q\subset4B_{j'}(x_i,2^{n'})\cap T.
\]
By Lemma~\ref{lem:ball-calculus}(iii) and
\eqref{eq:algorithm-descendant-cone},
\[
 4B_{j'}(x_i,2^{n'})
 \subset4B_{j^*+2}(x_i,2^{n^*+\kappa}),
\]
so $A'_q\subset A_i$. Give these sets the parameters
\[
 \mathbf n(A'_q):=n',
 \qquad
 \mathbf j(A'_q):=j',
\]
and put
\[
 g(A'_q):=f_{n'+\kappa,j'+2}(x'_q).
\]
Lemma~\ref{lem:ball-calculus}(iv) gives
\[
 \varphi_{j'+2}(A'_q,A'_{q'})>2^{n'}
 \qquad(q\ne q').
\]
Moreover, $n'\ge n^*+\kappa>n^*$. Thus the parameters $\mathbf n$ increase
along branches. The sets $A'_q$ are non-empty, pairwise disjoint, and
contained in $A_i$; hence $|A_i|\ge N_{n'}/2$, and $A_i$ has exactly
$N_{n'}/2$ children.

Repeat the construction below every non-terminal node. At each step the
$\mathbf j$-parameter increases by at least two, while all indices occurring
in the recursion are at most $J_+$. Hence the construction stops after
finitely many steps and gives a separation tree $\mathcal T$ on $T$.

By construction, $g(A)=0$ for every terminal node $A$. If $A$ is
non-terminal and $B_1,\ldots,B_m$ are its children, the corresponding
identity \eqref{eq:algorithm-split-identity} gives
\[
 g(A)
 =2^{\mathbf n(B_1)}r^{-\mathbf j(B_1)}
 +\min_{1\le i\le m}g(B_i),
\]
since all children of $A$ have the same parameters. It follows by induction
on the height of the tree that, for every $t\in S_{\mathcal T}$,
\[
 \sum_{\substack{A\in\mathcal T\setminus\{T\}\\t\in A}}
 2^{\mathbf n(A)}r^{-\mathbf j(A)}
 \ge F_*.
\]
Therefore
\[
 r^{-j_0}+F_*
 \le\operatorname{Size}(\mathcal T)
 \le\rho_\varphi(T).
\]
This proves the first assertion. The second follows from
Theorem~\ref{thm:algorithm-upper} and Theorem~\ref{m-2}.
\end{proof}

\subsection{Finitely many scales}

\begin{lemma}\label{lem:active-packing-scales}
If $A_{n,j}(u)>0$, then
\[
 j\in\mathcal J
 \qquad\text{and}\qquad
 j+2\in\mathcal J.
\]
\end{lemma}

\begin{proof}
If $A_{n,j}(u)>0$, then $n\ge1$, and the construction gives at least two
points $x_1,x_2\in2B_j(u,2^n)$ such that
\[
 x_2\notin9B_{j+2}(x_1,2^{n+\kappa}).
\]
If both $x_1$ and $x_2$ belonged to $2B_{j+2}(u,2^n)$, then
\eqref{eq:ball-arithmetic} would give
\[
 x_1-x_2\in4B_{j+2}(2^n)
 \subset9B_{j+2}(2^{n+\kappa}),
\]
a contradiction. Hence, for one of the two points, say $x_i$,
\[
 x_i\in2B_j(u,2^n),
 \qquad
 x_i\notin2B_{j+2}(u,2^n).
\]
Since the sets $2B_q(u,2^n)$ decrease with $q$, either $j\in\mathcal Q$ or
$j+1\in\mathcal Q$. The definition of $\mathcal J$ then gives both
$j\in\mathcal J$ and $j+2\in\mathcal J$.
\end{proof}

\begin{theorem}\label{thm:finite-algorithm}
Assume, in addition, that for given $x,y\in T$, $j\in\mathbb Z$, and
$1\le n\le n_T$, the relations
\[
 y\in2B_j(x,2^n),
 \qquad
 y\in9B_{j+2}(x,2^{n+\kappa})
\]
can be determined in time polynomial in $|T|$. Assume also that whether
\[
 \{j\in\mathbb Z:y\in2B_j(x,2^n)\}
\]
has a largest element and, if so, that element can be determined in time
polynomial in $|T|$. Then
\[
 V(T):=r^{-j_0}
 +\sup_{u\in T,\,n\ge0,\,j\ge j_0}f_{n,j}(u)
\]
can be computed in time polynomial in $|T|$. Moreover,
\[
 V(T)\asymp_r\gamma_\varphi(T)
 \asymp_r\mathcal M_\varphi(T)
 \asymp_r\rho_\varphi(T).
\]
\end{theorem}

\begin{proof}
The set $\mathcal Q$, and hence $\mathcal J$, can be determined by considering
all $x,y\in T$ and $1\le n\le n_T$. Since
\[
 2^{2^{n_T}-1}\le |T|,
\]
we have
\[
 n_T\le\log_2(1+\log_2|T|)
\]
when $n_T\ge1$. Thus \eqref{eq:active-scale-count} gives a polynomial bound
for $|\mathcal J|$.

By Lemma~\ref{lem:active-packing-scales},
\[
 A_{n,j}(u)>0
 \quad\Longrightarrow\quad
 j,j+2\in\mathcal J.
\]
In particular, if $j\notin\mathcal J$, then $A_{n,j}(u)=0$ for every $u$.
For such $j$, \eqref{eq:algorithm-recursion} contains only terms with larger
second index. Repeated use of this identity and the transitivity of
\eqref{eq:algorithm-order} show that indices outside $\mathcal J$ may be
omitted until either an index in $\mathcal J$ or the boundary $J_+$ is
reached. Hence it is enough to evaluate the functions $f_{n,j}$ for
$j\in\mathcal J$, together with the zero values for $j>J_+$.

For $j\in\mathcal J$, if $j+2\notin\mathcal J$, then
$A_{n,j}(u)=0$ by Lemma~\ref{lem:active-packing-scales}. Otherwise the
functions $f_{n+\kappa,j+2}$ have already been determined when the indices
$j$ are considered in decreasing order. Thus every $A_{n,j}(u)$ is obtained
from at most $|T|$ successive choices among points of $T$.

By \eqref{eq:active-scale-count}, the number of triples $(u,n,j)$ that need
to be considered is at most $|T|(n_T+1)|\mathcal J|,$
which is polynomial in $|T|$. For each triple, the construction of
$A_{n,j}(u)$ and the maximum in \eqref{eq:algorithm-recursion} involve only
polynomially many of the relations in the statement of the theorem.
Therefore $V(T)$ can be computed in time polynomial in $|T|$.

The comparison with the three geometric functionals follows from
Theorem~\ref{thm:algorithm-upper},
Theorem~\ref{thm:algorithm-tree-certificate}, and Theorem~\ref{m-2}.
\end{proof}

For the regular log-concave canonical processes of Section~\ref{5}, whenever
the assumptions of Theorem~\ref{thm:finite-algorithm} are satisfied,
Theorem~\ref{canonical processes} and Theorem~\ref{thm:finite-algorithm} give
a constant-factor deterministic approximation to $
 \textsf E\sup_{t\in T}X_t$.

\textbf{Acknowledgments}.
The work of Hanchao Wang was supported by the National Key R\&D Program of China (No.2024YFA1013501), the National Natural Science Foundation of China (No. 12571162), and Shandong Provincial Natural Science Foundation (No. ZR2024MA082).

\textbf{Authors' contributions}

These authors contributed equally to this work.

\textbf{Data availability}

No data were used in the research described in this article.

\printbibliography

\end{document}